\documentclass[a4paper, 10pt, final]{article}

\usepackage[utf8]{inputenc}			 % codifica di input
\usepackage[english]{babel}
\usepackage{asymptote}

\usepackage{amsmath}
\usepackage{amssymb}
\usepackage{amsthm}
	\theoremstyle{plain}
		\newtheorem{mainthm}{\textsc{Theorem}}		
						\newtheorem{thm}{Theorem}[section]
		\newtheorem{cor}[thm]{Corollary}	
			
		\newtheorem{lem}[thm]{Lemma}		
				\newtheorem{prop}[thm]{Proposition}

	\theoremstyle{definition}
		\newtheorem{defn}[thm]{Definition}					
			\theoremstyle{remark}
		\newtheorem{rem}[thm]{Remark}		
				\newtheorem{note}[thm]{Notation}

\numberwithin{equation}{section}	% equation numbering

\usepackage{mathtools}		% per i simboli := e =:
\usepackage{mathrsfs}		% per il math script
\usepackage{eucal}			% per il mathcal + figo
\usepackage{libertine}
\usepackage{braket}			% per il comando \Set: definizione
\usepackage[all,pdf]{xy}		% per i diagrammi commutativi
\usepackage{mparhack}		% fix bug for \marginpar
\usepackage{relsize}			% dimensioni testo

\usepackage{geometry}
\usepackage{booktabs}		% per fare bene le tabelle
\usepackage{multirow}		% per celle su piÂ righe nelle tabelle
\usepackage{caption}		% per le didascalie di tabelle e figure
\usepackage{rotating}
\usepackage{subfig}

\usepackage{varioref}		% per indicazione di pagina nei riferimenti
\usepackage{footmisc}

\usepackage{graphicx}		% per le figure
\usepackage{epsfig}
\usepackage{xcolor}
\usepackage{enumerate}		% per impostareÂ facilmente gli elenchi numerati
\usepackage[colorlinks=true, linkcolor=black, citecolor=black,
urlcolor=blue]{hyperref}		
\mathtoolsset{showonlyrefs}
\newcommand{\F}{\mathbb{F}}	
\newcommand{\K}{\mathbb{K}}	
\newcommand{\N}{\mathbb{N}}	
\newcommand{\R}{\mathbb{R}}	

\renewcommand{\S}{\mathbb{S}}
\newcommand{\C}{\mathbb{C}}		% the complex plane
\newcommand{\U}{\mathbb{U}}		% the unitary circle in the 
\newcommand{\Sp}{\mathrm{Sp}}
\newcommand{\diag}{\mathrm{diag}}
\newcommand{\X}{\mathcal{X}}

\newcommand{\Lin}{\mathscr{L}}

\newcommand{\norm}[1]{\left\| #1 \right\|}			% norma
\newcommand{\abs}[1]{\left\lvert #1 \right\rvert}	
\newcommand{\trasp}[1]{{#1}^\mathsf{T}}			
\newcommand{\iMor}{\mathrm{n_-}}		% indice di Morse
\newcommand{\coiMor}{\mathrm{n_+}}		% indice di Morse
\newcommand{\Mat}{\mathrm{Mat}}
\newcommand{\traspinv}[1]{{#1}^\mathsf{-T}}
\newcommand{\Gscal}[1]{\left\langle{#1}\right\rangle_\Gamma}
\newcommand{\Gnorm}[1]{\norm{#1}_\Gamma}
			\renewcommand{\leq}{\leqslant}
\renewcommand{\geq}{\geqslant}
\renewcommand{\=}{\coloneqq}			% definisce :=
\newcommand{\email}[1]{\href{mailto:#1}{\textsf{#1}}}
\usepackage{bm}
\newcommand{\simbolovettore}[1]{{\boldsymbol{#1}}}

\newcommand{\vC}{\simbolovettore{C}}
\newcommand{\vc}{\simbolovettore{c}}

\newcommand{\vs}{\simbolovettore{s}}

\newcommand{\vu}{\simbolovettore{u}}
\newcommand{\vv}{\simbolovettore{v}}

\newcommand{\vw}{\simbolovettore{w}}
\newcommand{\vx}{\simbolovettore{x}}

\newcommand{\vz}{\simbolovettore{z}}

\newcommand{\vxi}{\simbolovettore{\xi}}
\newcommand{\veta}{\simbolovettore{\eta}}
\newcommand{\Id}{I}
\renewcommand{\X}{\mathcal{X}}               %% configuration space

\title{Morse index and stability of the planar N-vortex problem}
\author{Xijun Hu\thanks{Partially supported by NSFC(Nos. 11790271, 11425105)}, Alessandro Portaluri
\thanks{Partially supported by Prin 2015: Variational methods, with applications to problems in mathematical physics and geometry},  Qin Xing}
\date{\today}

\date{\today}
\begin{document}
 \maketitle

\begin{abstract}

This paper concerns the investigation of the stability properties of relative equilibria which are rigidly rotating vortex configurations sometimes called vortex crystals, in the N-vortex problem. Such a configurations can be characterized as critical point of the Hamiltonian function restricted on the constant angular impulse hypersurface in the phase space (topologically a pseudo-sphere whose coefficients are the circulation strengths of the vortices). Relative equilibria are generated by the circle action on the so-called shape pseudo-sphere (which generalize the standard shape sphere appearing in the study of the  N-body problem).  Inspired by the planar gravitational $N$-body problem,  and after a geometrical and dynamical discussion, we investigate the  relation intertwining the stability of relative equilibria and the inertia indices of the central configurations generating such  equilibria.  In the last section we apply our main results to some symmetric three and four vortices relative equilibria.  

\vskip0.5truecm 
\begin{center}
\large{\em{In memory of Florin Diacu}	}
\end{center}

\end{abstract}

%%%%%%%%%%%%%%%%%%%%%%%%%%%%%%%%%%%%%%%%%%%%%%%%%%%%%%%%%%%%
%%%%%%%
%%%%%%%
%%%%%%%
%%%%%%%%%%%%%%%%%%%%%%%%%%%%%%%%%%%%%%%%%%%%%%%%%%%%%%%%%%%%

\section{Introduction and description of the problem}\label{sec:intro}

The study of vortex dynamics can be traced back to Helmholtz’s work on hydrodynamics in 1858 \cite{Hel1858} and it plays an important role in the study of   superfluids, superconductivity, and stellar systems. Its Hamiltonian formulation could be dated back to Kirchhoff in the plane, and later on generalized by Routh in \cite{Rou1880}  and then Lim \cite{Lim43} to general domains in the plane. In this paper, we  are interested to the problem in the first order Hamiltonian system of the form
\begin{equation}\label{eq:original-equation-intro}
\Gamma_{i}\dot{\vz}_{i}(t)=J\nabla_{\vz_i}H(\vz(t)) \qquad  i \in 1, \dots, N.
\end{equation}
Here $J\= \begin{bmatrix}
0&1\\
 -1&0
 \end{bmatrix}$ is the standard symplectic matrix in the Euclidean plane. The Hamiltonian function $H$ is 
\[
H(\vz)= - \sum_{\substack{i,j=1\\i<j}}^N \Gamma_i \Gamma_j \log \abs{\vz_i-\vz_j}.
\]
Here $\Gamma_1, \dots \Gamma_N \in \R\setminus \{0\}$ are the {\em vorticities\/} or {\em vortex strengths\/}. The Hamiltonian it is defined on the configuration space 
\begin{equation}
\F_N(\R^2)\=\Set{\vz \in \R^{2N}|  \vz_i \neq \vz_j \textrm{ for } i \neq j}
\end{equation}
of the $N$ (coloured) points in the plane. It is clear by the definition that $H(\vz_1, \dots, \vz_N)$ becomes singular if $\abs{\vz_i-\vz_j}\to 0$ for some $i \neq j$. Setting  $
G(\vw_1, \vw_2)=-\log\abs{\vw_1-\vw_2}$  then the Hamiltonian can also be written as $H(\vz)=\sum_{i<j}\Gamma_i\Gamma_j G(\vz_i, \vz_j)$ and it is usually called hydrodynamic Green's function. As already observed, the Hamiltonian system  in Equation \eqref{eq:original-equation-intro} appear as singular limit equations in problems from physics. More precisely in fluid dynamics is derived from the Euler equation and for instance in superconductivity $H$ appears as renormalized energy for Ginzburg-Landau vortices. Concerning the existence and stability properties of periodic solutions of the N-vortex problem given in Equation \eqref{eq:original-equation-intro},  the literature is quite broad  and we refer the interested reader to \cite{New01} and references therein. Among the simplest periodic orbits of the planar $N$-vortex problem are the {\em relative equilibria\/}. These  configurations of vortices rotates rigidly about their center of vorticity and sometimes are referred to as {\em vortex crystals\/} and are frequently observed in natural phenomena (e.g. the hurricanes).

 Relative equilibria are crucial in deeply understanding the intricate dynamics of this singular Hamiltonian problem and as the name suggested are rest-points  in a suitable rotating coordinate system. As we'll discuss in Section \ref{sec:The geometrical and dynamical framework}, relative equilibria can be characterized as critical points (or more precisely critical orbits) of the restriction of the Hamiltonian to the angular impulse unitary (pseudo-)sphere of the phase space. Otherwise stated such a rigid configurations are generated through a rotation with angular velocity $\omega$ of a special critical configuration of the system usually called {\em central configuration\/}. (Cfr. Section \ref{sec:The geometrical and dynamical framework} for further details). 
  
    A natural and classical problem is to understand how the spectral properties of these central configurations or more precisely the inertia indices of the Hessian of $H$ at these configurations reflect on the dynamical properties of the generated vortex crystal (through rotation) like, for instance, spectral or linear stability properties etc. This problem is very classical in the gravitational $N$-body problem in which central configurations are characterized as critical points of the self-interacting potential on the shape sphere (which is the base space of the circle bundle) whose total space is the inertia ellipsoid. There is a long standing conjecture due to Moeckel stating that a linearly stable relative equilibrium must be a nondegenerate minimum of the Newtonian potential restricted to the shape sphere. The other direction is false even for other class of weakly attracting singular potentials. (Cfr., for instance, \cite{HLS14, BJP16}). 
    
       The investigation of the relation between the stability properties of a relative equilibrium and the spectral properties of the central configuration generating such an equilibrium in the N-vortex problem is pretty different. Despite of the fact that the circulation strengths could have any sign (in the classical gravitational $N$-body problem they correspond to the masses which are all positive), the Hessian of the Hamiltonian  computed at a central configuration has some commutativity properties with respect to the Poisson matrix $K$ induced by  $J$ that greatly simplify the problem. (Cfr. Lemma \ref{thm:property-of-H}, for further details). Such a property was  observed by  Roberts  in its interesting  paper of  Roberts  in \cite{Rob13}. In the aforementioned paper, in fact, the author was able among others to characterized in the case of positive circulation strengths, the linearly stable relative equilibria of the $N$-vortex problem as nondegenerate minima of the Hamiltonian $H$ restricted to the shape sphere. This result was the starting point of our analysis and motivated us to investigate what is the effect of mixed sign circulation strengths, which after all, are very common  in the applications. 
       
          Before describing our main results we start to observe that this indefinite case,  the stability analysis is much more delicate. This situation, as we'll try to clarify, reflects somehow the difficulties and it is the paradigm of the difference between the Riemannian and the Lorentzian world.

    \subsection{Main results}\label{subsec:main-results}
        Our first result, provides a characterization of the spectral stability  of  a relative equilibrium $\vz$  in terms of a spectral condition on the central configuration $\vxi$, no matter how the signs of the circulations are. Before stating and describing our first result, we pause by recalling what stability notion we are talking about. Being the Hamiltonian $H$ invariant under translations and rotations this implies, among others, that $0$ (having algebraic multiplicity $2$) as well as $\pm\omega i$ are Floquet characteristic multipliers, arising precisely from these symmetries. 
      
      What is natural to do is  to define the linear stability properties of a relative equilibrium by ruling out the eigenvalues coming from these conservation laws. More precisely, we define linear stability by restricting to a complementary subspace of the invariant space defined by the above symmetries. 
\begin{defn}\label{def:stability}
A relative equilibrium $\vz$ will be termed {\em non-degenerate\/} provided the remaining $2n-4$ eigenvalues of the matrix $B$ are not vanishing. A non-degenerate relative equilibrium is 
\begin{itemize}
\item {\em spectrally stable\/} if the nontrivial eigenvalues are purely imaginary 
\item {\em linearly stable\/} if, in addition, the restriction of the stability
matrix $B$ to $W^{\perp}$ has a block-diagonal Jordan form with
blocks $\begin{bmatrix}0 & \beta_{i}\\
-\beta_{i} & 0
\end{bmatrix}.$ 
\end{itemize}
\end{defn}
\begin{rem}
Otherwise stated a  relative equilibrium $\vz$ is spectrally stable if all Floquet multipliers (the eigenvalues of the monodromy matrix) belongs to the unit circle  (centered at the origin) $\U$ in the complex plane. Furthermore if the monodromy matrix is also diagonalizable, then $\vz$ is linearly stable. In this last case, in fact, the monodromy matrix can be factorized as direct symplectic sum of rotations or which is the same, it belongs to the maximal compact Lie subgroup  of $\Sp(2N-4)$.
\end{rem}

                \begin{mainthm}\label{thm:generalization-3-1-rob-intro}
A non-degenerate relative equilibrium $\vz$ (with angular velocity $\omega$) generated by the central configuration $\vxi$ is spectrally stable
if and only if for every eigenvalue $\mu$ of  $M_\Gamma^{-1}D^{2}H(\vxi)$ one of the following alternative holds
\begin{itemize}
\item $\mu \in i\R$ 
\item $\mu \in \R$ and $\mu \in [-|\omega|, |\omega|]$
\end{itemize}
where $M_\Gamma\= \Gamma_i \, \Id_2\delta_{ij}$ (where $\Id_2$ denotes the $2\times 2$ identity matrix and $\delta_{ij}$ is the Kronecker delta). In particular it is non-degenerate if and only if $\mu \neq \pm \omega$.
\end{mainthm}
The idea for proving this result is essentially based on the relation between the matrix $M_\Gamma^{-1}D^{2}H(\vxi)$ and the so-called {\em stability matrix\/}, namely the matrix which is responsible of the stability of the relative equilibrium which is defined by $B= K\big[M_\Gamma^{-1}D^{2}H(\vxi)+\omega\Id_{2N}\big]$ where $K$ is the Poisson matrix, i.e. a $2\times 2$ block diagonal matrix in which each non-vanishing block is given by $J$. 

It is worth to observe that, no matter how the signs of the circulations are,  the matrix appearing in Theorem  \ref{thm:generalization-3-1-rob-intro}
 i.e. $M_\Gamma^{-1}D^2H(\vxi)$ is $M_\Gamma$-symmetric namely is symmetric with respect to the scalar product induced by $M_\Gamma$ (cfr. Definition \ref{def:N-symmetry} for further details). However, if the circulations are all positive such a scalar product is positive definite and this implies that the spectrum of $M_\Gamma^{-1} D^2H(\vxi)$ is diagonalizable in the orthogonal group and its spectrum is real.  For mixed signs circulations, this is not true anymore, and in fact the spectrum of $M_\Gamma^{-1} D^2H(\vxi)$ will be, in general,  not real anymore.  This reflects  the indefinite Krein structure behind and among others responsible of the presence of Jordan blocks that are intimately related to the spectral stability properties of the relative equilibrium. In conclusion Theorem \ref{thm:generalization-3-1-rob-intro} represents the generalization of \cite[Theorem 3.1]{Rob13} in the case of mixed signs circulations. The matrix $B$ is $M_\Gamma$-Hamiltonian, namely is Hamiltonian with respect to the symplectic form $\omega_\Gamma$ which is represented by $K$ with respect to the $M_\Gamma$-scalar product. 
 \vskip0.5truecm
 Our second main result relates the spectrally stability properties of a relative equilibrium and the inertia indices of the central configuration generating it. 
  \begin{mainthm}\label{thm:stability-Morse-index-intro}
Let $\vz$ be non-degenerate relative equilibrium generated by the central configuration $\vxi$ and let $\widehat A(\vxi)\= D^2H(\vxi)+\omega M_\Gamma$. We assume that $\vz$ is spectrally stable. Then the following result holds.
\begin{itemize}
\item \underline{Case of positive angular velocity $\omega$} 
\begin{equation}\label{eq:index formula-1}
\begin{cases}
\iMor(\widehat{A}(\vxi))=\iMor(M_\Gamma), & \textrm{ if  } \left\langle M_\Gamma\vxi,\vxi\right\rangle \textrm{ is positive definite }\\
\iMor(\widehat{A}(\vxi))=\iMor(M_\Gamma)-1, & \textrm{ if }    \left\langle M_\Gamma\vxi,\vxi\right\rangle \textrm{ is negative definite }.
\end{cases}
\end{equation}
\item  \underline{Case of negative angular velocity $\omega$} 
\begin{equation}\label{eq:index formula-2}
\begin{cases}
\iMor(\widehat{A}(\vxi))=\coiMor(M_\Gamma)-1, & \textrm{ if } \left\langle M_\Gamma\vxi,\vxi\right\rangle \textrm{ is positive definite }\\
\iMor(\widehat{A}(\vxi))=\coiMor(M_\Gamma), & \textrm{ if }  \left\langle M_\Gamma\vxi,\vxi\right\rangle \textrm{ is negative definite }.
\end{cases}
\end{equation}
\end{itemize}
Furthermore, we have 
\begin{equation}\label{eq: restriction index formula}
\begin{cases}
\iMor(\widehat{A}|_{W^{\perp}}(\vxi))=\iMor(M_\Gamma|_{W^\perp}), & if\ \omega>0,\\
\iMor(\widehat{A}|_{W^{\perp}}(\vxi))=\coiMor(M_\Gamma|_{W^\perp}), & if\ \omega<0
\end{cases}
\end{equation}
where $W^\perp$ is the $M_\Gamma$-orthogonal complement of   $W=\textup{span}(\vxi, K\vxi)$.
\end{mainthm}
Given $N$-vortices in the plane, we define {\em total vortex angular momentum\/} $L$  as $L\= \sum_{i<j} \Gamma_i \Gamma_j$. Thus, if the vortex strengths are all positive, then $L>0$. However, when vorticities are different in signs, then $L$ could be of any sign  or even vanishes. Analogously to the moment of inertia in the $N$-body problem, it is possible to define the so-called {\em angular-impulse of the $N$-vortex problem\/} as follows
\begin{equation}
I(\vz)\= \dfrac12 \sum_{i=1}^N \Gamma_i \norm{\vz_i}^2
\end{equation}
and, as we will see  in the sequel, it will be crucial in order to give a variational interpretation to the central configurations, miming the analogous interpretation in the $N$-body problem. 
As already observed, for a relative equilibrium $\vz$ generated by $\vxi$, the angular velocity is constant and it is given by 
\begin{equation}\label{eq:segno-omega-intro}
	L=\omega \, \langle \nabla I(\vxi), \vxi \rangle= 2\, \omega \, I(\vxi)\quad \Rightarrow \quad  \omega=L/\big(2\, I(\vxi)\big).
\end{equation}
  Thus in the case of mixed signs circulations the angular velocity could have any sign what that cannot happen in the case of constant sign circulations. 
  
  The proof of Theorem \ref{thm:generalization-3-1-rob-intro} and Theorem \ref{thm:stability-Morse-index-intro} will be given in Section  \ref{sec:Proof-Main-results-abstract}. 
 
 In the last section, we analyze some interesting symmetric central configuration; more precisely, the equilateral triangle and the rhombus (sometimes called {\em kite\/}) central configuration. 
 
The equilateral triangle central configuration in the three-vortex problem is obtained by placing three vortices of any strength at the vertices of an equilateral triangle. Synge in his celebrated paper  published in 1949 (cfr. \cite{Syn49} ), proved that the corresponding relative equilibrium is  linearly stable if and only if $L>0$.  

Starting from this we get information on the a central configuration  knowing the stability of the induced relative equilibrium. More precisely, let us given three circulations $\Gamma_{1},\ \Gamma_{2},\ \Gamma_{3}$ placed at the following points
\[
\widehat{\vxi}_{1}=(1,0),\ \widehat{\vxi}_{2}=(-\dfrac{1}{2},\dfrac{\sqrt{3}}{2}),\ \widehat{\vxi}_{3}=(-\dfrac{1}{2},-\dfrac{\sqrt{3}}{2}),
\]
and we let $\widehat{\vc}=\sum_{i=3}^{3}\Gamma_{i}\widehat{\vz}_{i}$. Assuming  that  $L>0$ and  setting $\vxi=(\vxi_{1},\vxi_{2},\vxi_{3})$, for  $\vxi_{i}=\widehat{\vxi}_{i}-\widehat{\vc}$ then we conclude that 
\[
\iMor(\widehat A_\Gamma(\vxi))=
\begin{cases} 0  & \textrm{ if } \Gamma_{1}, \Gamma_{2}, \Gamma_{3} \textrm { have the same sign }\\
1  & \textrm{ if there is only one } \Gamma_{i}<0\\
2  & \textrm{ otherwise }
\end{cases}.
\]
About the kite central configuration, it is know (cfr. \cite{Rob13} for further details) that there exist two families of
relative equilibria where the configuration is a rhombus. Set $\Gamma_{1}=\Gamma_{2}=1$
and $\Gamma_{3}=\Gamma_{4}=m,$ where $m\in(-1,1]$ is a parameter.
Place the vortices at $z_{1}=(1,0)$, $z_{2}=(-1,0)$, $z_{3}=(0,y)$
and $z_{4}=(0,-y)$, forming a rhombus with diagonals lying on the
coordinate axis. This configuration is a central configuration provided that 
\begin{equation}\label{eq:rhombus-the-relationship-of-y-and-m-intro}
y^{2}=\dfrac{1}{2}\left(\beta\pm\sqrt{\beta^{2}+4m}\right),\ \ \beta=3(1-m).
\end{equation}
The angular velocity is given by 
\[
\omega=\dfrac{m^{2}+4m+1}{2(1+my^{2})}=\dfrac{1}{2}+\dfrac{2m}{y^{2}+1}.
\]
Taking plus sign in Equation \eqref{eq:rhombus-the-relationship-of-y-and-m-intro} yields a solution for $m\in(-1,1]$ that always has $\omega>0$. We refer to this solution as 
rhombus $A$. Taking $-$ in Equation \eqref{eq:rhombus-the-relationship-of-y-and-m-intro} yields a solution for $m\in(-1,0)$ having  $\omega>0$ for $m\in(-2+\sqrt{3},0),$
but $\omega<0$ for $m\in(-1, -2+\sqrt{3})$. We refer to this solution as rhombus $B$. Assuming  that $\vz$ is the relative equilibrium generated by the rhombus  central configuration $\vxi$. Then 
\begin{enumerate}
\item if the central configuration is  rhombus $A$, then we have 
\[
\iMor(\widehat{A}_\Gamma(\vxi))=\begin{cases}
0 & \textrm{ if }\quad 0<m\le1,\\ 
3 & \textrm{ if }\quad -2+\sqrt{3}<m<0,\\
4 & \textrm{ if }\quad -1<m<-2+\sqrt{3}.
\end{cases}
\]
\item if the central configuration is  rhombus $B$, then we have 
\[
\iMor(\widehat{A}_\Gamma(\vxi))=
\begin{cases}
%0+0+0+1+1=
2 & \textrm{ if }\quad -2+\sqrt{3}<m<0,\\
%0+0+2+1+1=
4 & \text{ if }\quad m^{*}<m<-2+\sqrt{3},\\
%0+0+2+1+0=
3 & \textrm{ if }\quad -1<m<m^{*}.
\end{cases}
\]
where $m^*$ is the only real root of the cubic $9m^3+3m^2+7m+5$.
\end{enumerate}

 The paper is organized as follows: 
 
 \tableofcontents

\subsection*{Notation}
At last, let us introduce some notation that we shall use henceforth without further reference. We have already mentioned that I stands for the {\em angular impulse\/} , however, the similar symbol $\Id_X$ or just $\Id$ will denote the identity operator on a space $X$ and we set for simplicity  $\Id_k := \Id_{\R^k}$ for $k \in \N$. We denote throughout by the symbol $\trasp{\#}$ (resp. $\traspinv{\#}$) the transpose (resp. inverse transpose) of the operator $\#$. \\
 $\Mat(m,n; \K)$ stands for  the space of  $m \times n$ matrices in the field $\K$ and if $m=n$ we just use the short-hand notation  $\Mat(m; \K)$. $\sigma(\#)$ denotes the spectrum of the linear operator $\#$.  We denote throughout by $J$ the standard  symplectic matrix $J\= \begin{bmatrix} 0&1\\-1&0 \end{bmatrix}$. $\U$ denotes the unit circle in the complex plane namely the set of all  complex numbers of modulus $1$.\\
 If  $Z$ is a finite dimensional vector space. We denote by $\Lin(Z)$ the vector space of all linear operators on $Z$.
%%%%%%%%%%%%%%%%%%%%%%%%%%%%%%%%%%%%%%%%%%%%%%%%%%%%%%%%%%%%
%%%%%%%
%%%%%%%
%%%%%%%
%%%%%%%%%%%%%%%%%%%%%%%%%%%%%%%%%%%%%%%%%%%%%%%%%%%%%%%%%%%%

\subsection*{Acknowledgements}
The second named author wishes to thank all faculties and staff of the Mathematics Department in the Shandong University (Jinan)  for providing excellent working conditions during his stay. 

%%%%%%%%%%%%%%%%%%%%%%%%%%%%%%%%%%%%%%%%%%%%%%%%%%%%%%%%%%%%
%%%%%%%
%%%%%%%
%%%%%%%
%%%%%%%%%%%%%%%%%%%%%%%%%%%%%%%%%%%%%%%%%%%%%%%%%%%%%%%%%%%%

\section{The geometrical and dynamical framework}\label{sec:The geometrical and dynamical framework}

In the Euclidean plane $(V, \langle \cdot, \cdot \rangle)$ equipped with coordinates $\vz=(p,q)$, we consider the standard symplectic form $\omega$ defined as follows 
\begin{equation}
\Omega(\cdot, \cdot) = \langle J\cdot , \cdot \rangle \quad \textrm{ where }\quad J\= \begin{bmatrix}
 0&1\\
 -1&0
 \end{bmatrix}.
\end{equation}
For  $i \in \{1, \dots, N\}$, let  $\Gamma_i\in \R\setminus\{0\}$ representing the vortex strength of the $N$-point vortex $\vz_i \in V$. We will assume throughout that the {\em total circulation\/} $\Gamma\= \sum_{i=1}^N \Gamma_i$ is nonzero. The  {\em center of vorticity\/} is then well-defined as  $\vc\=\Gamma^{-1}\sum_{i=1}^N \Gamma_i\vz_i$.  Let  $H: V^N\to \R$ be the function defined as follows
\begin{equation}\label{eq:Hamiltonian}
H(\vz)= - \sum_{i<j} \Gamma_i \Gamma_j \log r_{ij} \quad \textrm{ where } \quad r_{ij}\=\abs{\vz_i-\vz_j}
\end{equation}
for $\vz= (\vz_1, \dots, \vz_N) \in V^N$. In what follows we refer  to $H$ as the {\em $N$-vortex Hamiltonian function\/}. Denoting by $\Id_N$ the $N\times N$ identity matrix, we define the {\em matrix of circulations\/}  as the real $2N \times 2N$ matrix given by
\begin{equation}
M_\Gamma \= \begin{bmatrix}
 \Gamma_1 & & \\
 & \ddots & \\
 & & \Gamma_N
 \end{bmatrix}\otimes \Id_2= 
\begin{bmatrix}
 \Gamma_1 \, \Id_2& & \\
 & \ddots & \\
 & & \Gamma_N\,\Id_2
 \end{bmatrix}\in\Mat(2N; \R)
\end{equation}
and the symplectic matrix $K\= \Id_N \otimes J \in \Mat(2N; \R)$.  Let $\F_N(V)$ be the space of all $N$ (colored) points in $V$; in symbols
\begin{equation}
\F_N(V)\=\Set{\vz \in V^N| i \neq j \Longrightarrow \vz_i \neq \vz_j}= V^N\backslash \Delta.
\end{equation}
Its complement in $V^N$ is the  {\em collision set\/}
\begin{equation}
\Delta\=\Set{\vz \in V^N|\exists\, (i,j), i \neq j: \vz_i = \vz_j}= \bigcup_{1 \leq i <j \leq N} \Delta_{ij}
\end{equation}
where $\Delta_{ij}\=\Set{\vz \in V^N| \vz_i= \vz_j}$. It is immediate to check that the restriction of $H$ to $\F_N(V)$ is indeed a smooth function.

%%%%%%%%%%%%%%%%%%%%%%%%%%%%%%%%%%%%%%%%%%%%%%%%%%%%%%%%%%%%
%%%%%%%
%%%%%%%
%%%%%%%
%%%%%%%%%%%%%%%%%%%%%%%%%%%%%%%%%%%%%%%%%%%%%%%%%%%%%%%%%%%%

\subsection{Central configurations}\label{subsec:Central configurations}

Given two vectors $\vv, \vw$  in   $V^N$, then we let 
\begin{equation}\label{eq:circulation-product}
\Gscal{\vv, \vw}\= \sum_{i=1}^N \Gamma_i \vv_i \cdot \vw_i
\end{equation}
denote the {\em circulation scalar product of $\vv$ and $\vw$\/}, where $\vv_i \cdot \vw_i$ denotes the standard Euclidean product in $V$ of the $i$-th component of $\vv$ and $\vw$.  We observe that, if the vortex strengths are all positive, then $\Gscal{\cdot, \cdot}$ is, actually, an inner product equivalent to the Euclidean one; otherwise, is an indefinite (non-degenerate) scalar product.\footnote{%
This indefinite scalar product appears very often in mathematical physics; e.g. the Minkowski scalar product.} We also notice that $\Gscal{\vv, \vw}= \trasp{\vw}M_\Gamma \vv$ where $\trasp{\cdot}$ denotes the transpose with respect to the Euclidean product. Given $I_0 \in \R$, we define the     pseudo-sphere  $\S_N(V)$  and we'll refer to as  {\em circulation (pseudo)-sphere\/} or {\em circulation sphere\/} for short as 
\begin{equation}\label{eq:circulation-sphere}  
\S\=\S_N(V)\=\Set{\vz \in \F_N(V)| \Gnorm{\vz}^2=I_0}.
\end{equation} 
%
% {\color{red} {Pleae, note the fact that if $M_\Gamma$ is indefinite, it can happen that $\Gnorm{\vz}^2<0$(See the Fourth case in the proof of Theorem 5.1).  Maybe we can define  $\S\=\S_N(V)\=\Set{\vz \in \F_N(V)| \Gnorm{\vz}^2=I_0}$ , where $I_0$} is a non-zero constant, what do you think? } 
%

In particular the circulation sphere is equal to the   sphere (with respect to the circulation scalar product) in $V^N$ with collisions removed; thus $\S_N(V)= S_N(V)\setminus \Delta$ where
\[
S_N(V) \= \Set{\vz \in V^N| \Gnorm{\vz}^2=I_0}.
\]
\begin{rem}
It is worth noticing that if vortex strengths are all positive, then the  pseudo-sphere is in general an ellipsoid (thus, topologically a sphere) and if are all equal it reduces to the round sphere. In the general case, however, it  is a (non-compact) quadric.
\end{rem}

A {\em central configuration for the the $N$-vortex problem\/} is a (non-collision) configuration $\vxi \in \F_N(V)$ with the property that exists $\omega \in \R$ such that 
\begin{equation}\label{eq:cc-vortex}
\nabla H(\vxi) + \omega \nabla I(\vxi)=\nabla H(\vxi) + \omega  M_\Gamma(\vxi)=0. 
\end{equation}
Let $\vC: V^N \to V^N$ be the isometry  defined by $\vC(\vxi)= \vxi'$, with 
\begin{equation}
\vxi'_j= \vxi_j -2 \vc
\end{equation}
for each $j=1, \dots, N$. It is easy to check that $\vC$ is an isometry with respect to the $\Gscal{\cdot, \cdot}$.  In fact, it holds
\begin{multline}\label{eq:isometria}
\Gnorm{\vxi'}^2= \sum_{i=1}^N \Gamma_i \Gnorm{\vxi_i - 2 \vc}^2= \sum_{i=1}^N \Gscal{\Gamma_i \vxi_i, \vxi_i}- 4 \sum_{i=1}^N\Gscal{ \Gamma_i \vxi_i, \vc} + 4 \sum_{i=1}^N \Gamma_i \Gnorm{\vc}^2\\
= \Gnorm{\vxi}^2 -4 \Gnorm{\vc}^2\Gamma + 4 \Gnorm{\vc}^2 \Gamma=\Gnorm{\vxi}^2.
\end{multline}
We observe that $H(\vC \vxi)=H(\vxi)$ and by the computation performed in Equation \eqref{eq:isometria}, we conclude immediately that $I(\vxi')=I(\vxi)$. By these two fact readily follows that  if $\vxi$ is a central configuration then also $\vC \vxi$ is a central configuration. Now, by using Equation \eqref{eq:cc-vortex}, we  can conclude that   $\vC \vxi=\vxi$  and hence $\vxi=0$. Thus, if $\vxi$ is a central configuration, then  its center of vorticity $\vc=0$.
As consequence of this discussion and without leading in generalities in the sequel we'll restrict to the {\em reduced  phase space\/} which is the $2(N-1)$-dimensional subspace of $V^N$ defined by 
\begin{equation}
\X\= \Set{\vz=(\vz_1, \dots, \vz_N) \in V^N| \sum_{i=1}^N \Gamma_i\vz_i=0}.
\end{equation}

By using once again Equation \eqref{eq:cc-vortex} it follows in fact, that a central configuration can be seen as a critical point of the Hamiltonian function restricted to a level surface of the angular impulse in which $\omega$ acts as a Lagrangian multiplier. Note that, if $\bar\vz$ is a central configuration, so is $\lambda \bar \vz$ for any scalar $\lambda$. In this case, the parameter $\omega$ must be scaled by a factor $1/\lambda^2$. 

We observe that by the  invariance property of the Hamiltonian function as well as of the angular-impulse, we get that central configurations are not isolated and appears in a continuous family. To eliminate such a degeneracy, it is customary to fix a scaling (e.g. $I=I_0$) and to identify central configurations that are identical under rotations. 
\begin{rem}
It is also worth noticing that the linear stability properties of such rigid motions are not affected by such rotation. 	
\end{rem}
We also define the following sets 
 \[
 \S_N^\vc\= \S_N \cap \X \quad \textrm{ and }\quad S_N^\vc\= S_N \cap \X.
\]  
By the above discussion, in particular we get that if $\vz$ is a central configuration then $\vz \in \X$. However, in principle, a critical point of the restriction of $H|_{\S_N^\vc}$ is not necessarily a critical point of $H|_{\S_N}$. However since the Hamiltonian function is $\vC$-invariant and being $\X$ the space fixed by the action of the (compact Lie) orthogonal  group of $V^N$, it follows that  any critical point of the restriction of $H|_{\S_N^\vc}$ is indeed a critical point of $H|_{\S_N}$ (cf. \cite{MR547524}, for further details). However, as already observed critical points of $H|_{\S_N^\vc}$ are not isolated. In fact, if $\vz_0$ is a critical point of $H|_{\S_N^\vc}$, then $e^{\vartheta K} \vz_0$ is,  for every $\vartheta$. In order to eliminate this further degeneracy, we consider the quotient spaces
\[
 \widehat \S_N^\vc\= \S_N^\vc/S^1 \quad \textrm{ and }\quad \widehat S_N^\vc\= S_N^\vc/S^1
\]  
and we'll refer to the {\em shape sphere without collision\/} and the {\em shape sphere\/} respectively. It is worth noticing that both are the orbit space of the circle action on the spheres $\S_N^\vc$ and $ S_N^\vc$, respectively. In what follows, we'll refer to a central configuration as the critical point of $H_{\widehat \S_N^\vc}$  in order to distinguish from  critical points of $H_{\S_N^\vc}$ usually called {\em relative equilibria\/}.

%%%%%%%%%%%%%%%%%%%%%%%%%%%%%%%%%%%%%%%%%%%%%%%%%%%%%%%%%%%%
%%%%%%%
%%%%%%%
%%%%%%%
%%%%%%%%%%%%%%%%%%%%%%%%%%%%%%%%%%%%%%%%%%%%%%%%%%%%%%%%%%%%

\subsection{Relative equilibria}

A system of $N$ point vortices (in the plane) with vortex strength $\Gamma_{i}\neq0$ and positions $\vz_{i}\in V$  evolves according to the phase flow induced by the following Hamiltonian  system
\begin{equation}\label{eq:original equation}
\Gamma_{i}\dot{\vz}_{i}=J\nabla_{i}H\big(\vz(t)\big)=J\sum_{j\neq i}^{N}\dfrac{\Gamma_{i}\Gamma_{j}}{r_{ij}^{2}}(\vz_{j}-\vz_{i}), \qquad  i \in \{1, \dots, n\}
\end{equation}
where the Hamiltonian function $H$ is defined in Equation \eqref{eq:Hamiltonian},  and $\nabla_{i}$ denotes the two-dimensional partial gradient with respect to $z_{i}$. We will assume throughout that
\begin{enumerate}
\item[(H)]  the {\em total circulation} $\Gamma = \sum_{i=1}^N \Gamma_i $ is nonzero and that the center of vorticity $\vc = 0$. 
\end{enumerate}
In short-hand notation Equation \eqref{eq:original equation}, could be rewritten in the following form
\begin{equation}\label{eq:equation-motion}
M_\Gamma \dot \vz(t)= K\nabla H\big(\vz(t)\big), \qquad t \in [0, 2\pi]
\end{equation}
where $K$ is  the real $2N \times 2N$ matrix given by

{  \begin{equation*}
	K\=I_N\otimes J=\begin{bmatrix}
	J& & \\
	& \ddots & \\
	& & J
	\end{bmatrix}\in\Mat(2N; \R).
	\end{equation*}  }
A special class (maybe the easiest) of periodic solutions for this problem is given by the rigid motions of the system around its  center of mass. Such a motions are termed relative equilibria. More precisely we introduce the following definition. 
\begin{defn}\label{def:RE}
We term {\em relative equilibrium\/} (RE, for short) any $T\=2\pi/|\omega|$ periodic solution  of Equation \eqref{eq:equation-motion}, namely
\begin{equation}\label{eq:RE}
\vz(t) \= e^{- \omega K t} \vxi, \qquad \textrm{ where } \omega \in \R\setminus \{0\},\quad t \in [0,T] \textrm{ and } \vxi \in \F_N(V).
\end{equation}
\end{defn}
\begin{rem}
By this definition, as already observed, it follows that a relative equilibrium is a periodic solution in which each point vortex uniformly rotates with angular velocity $\omega \neq 0$ around (its common center of vorticity represented by) the origin.
\end{rem}
By a direct computation and by using Equation \eqref{eq:RE}, it follows that the central configuration $\vxi$ generating  a relative equilibrium satisfy the following equation
\begin{equation}\label{eq:initial equation}
-\omega\,\Gamma_{i}\vxi_{i}\;=\;\nabla_{i}H\big(\vxi\big)\;=\;\sum_{j\neq i}^{N}\dfrac{\Gamma_{i}\Gamma_{j}}{r_{ij}^{2}(0)}\big(\vxi_{j}-\vxi_{i}\big),\quad \textrm{ for each }i\in\{1, \dots, n\}.
\end{equation}
Otherwise said, for every $i \in \{1, \dots,n\}$,  $\nabla_{i}H\big(\vxi\big)+ \omega\,\Gamma_{i}\vxi=0$ which is equivalent to claim that $\vxi$ is a solution of Equation \eqref{eq:cc-vortex} hence a central configuration. Thus in a properly rotating frame a relative equilibrium is nothing but a central configuration. 

The following result points out some crucial properties of the Hamiltonian function $H$ that will be useful later on and we refer the interested reader to  \cite[Lemma 2.3]{Rob13} for the proof. 
\begin{lem}\label{thm:property-of-H}
The Hamiltonian $H$ has the following
three properties:
	\begin{enumerate} 
		\item[{\bf (i)}] $ \nabla H(\vz) \cdot \vz  \; = \;  - L, $
\item[{\bf (ii)}] $\nabla H(\vz) \cdot (K \vz) \; = \; 0$,
\item[{\bf (iii)}]  $D^2H(\vz) \, K \; = \;  - K \, D^2H(\vz)$.
\end{enumerate}
\end{lem}
\begin{rem}
	As we will see later on, property (iii) plays a crucial role in the investigation of the  linear stability for relative equilibria. For all of the same  sign vorticity strengths, such a condition reduces the problem to the investigation of the spectrum to a $2 \times 2$ symmetric matrix or equally well a complete factorization of the characteristic polynomial into even quadratic factors. 
	
	This property doesn't hold for relative equilibria of the $N$-body problem and in fact a challenging longstanding still open problem is to establish a precise relation between  the dynamical properties of the  relative equilibria and the 
  spectral properties of central configurations originating them.
  \end{rem}
Differentiating with respect to $\vz$ the equality appearing  at first item in Lemma \ref{thm:property-of-H},  we get 
\begin{multline}\label{eq:uu}
\langle D^2H(\vxi)[\vu], \vz\rangle + \langle \nabla H(\vxi), \vu\rangle = \langle D^2H(\vxi)[\vz], \vu\rangle + \langle \nabla H(\vxi), \vu\rangle =0, \quad \forall\, \vu \in T_\vxi  \widehat \S_N^\vc\\  \Rightarrow  \qquad  D^2H(\vxi)[\vz]+ \nabla H(\vxi)=0.
\end{multline}
Since $\vxi$ is a  central configuration and by using once again Equation \eqref{eq:cc-vortex}, we immediately get 
	$\nabla H(\vxi)=-\omega M_\Gamma \vxi$ and by summing up we get the equality 
		\begin{equation}\label{eq:impRob}
	M_\Gamma^{-1}D^{2}H(\vxi)\,\vxi =\omega \, \vxi.
	\end{equation}
	Equation \eqref{eq:impRob} together with property (iii) Lemma \ref{thm:property-of-H} shows that $BK\vx=0$.
The equation of motions given in Equation \eqref{eq:equation-motion}, in a uniformly rotating frame with angular velocity $\omega$ reduces to 
\begin{equation}\label{eq:equation-motion-rotating}
	M_\Gamma \dot \vw(t)= K \big(\nabla H(\vw(t))+\omega M \vw(t)\big), \qquad t \in \left[0, \dfrac{2\pi}{|\omega|}\right].
\end{equation}
In fact, let $\vw(t)\= e^{\omega K t}\vz(t)$; thus by a direct computation, we get
\begin{equation*}
	M_\Gamma \dot \vw(t)=\omega KM_\Gamma e^{\omega Kt}\vz(t)+M_\Gamma e^{\omega Kt} \nabla H(\vz(t))= K \big[\nabla H\big(\vw(t)\big)+\omega M_\Gamma \vw(t)\big] 
\end{equation*}
where the commutativity properties of $M_\Gamma$ with respect to $K$ and $e^K$ were tacitly used. 
In particular, a rest point of the Hamiltonian vector field appearing in Equation \eqref{eq:equation-motion-rotating} is a relative equilibrium, as expected.
\begin{rem}
We observe that if the circulations have mixed sign then $\omega$ could be of any sign (meaning that the vortices can rotate clockwise or counterclockwise with respect to the center of vorticity). In fact, by taking the scalar product with respect to $\vxi$ in Equation \eqref{eq:cc-vortex} as well as invoking the first claim in Lemma \ref{thm:property-of-H}, we get that 
\begin{equation}\label{eq:segno-omega}
	L=\omega \, \langle \nabla I(\vxi), \vxi \rangle= 2\, \omega \, I(\vxi)\quad \Rightarrow \quad  \omega=L/\big(2\, I(\vxi)\big).
\end{equation}
where the last equality directly follows by using the Euler theorem on positively homogeneous functions after observing that $I$ is homogeneous of degree $2$. Now, the claim follows by observing that a priori $L$ could be of either positive or negative. 
\end{rem}
\begin{rem}
In a more geometrical way the Hamilton equations in the uniformly  rotating frame are nothing but the Hamilton equation on the cotangent bundle $T^*S_N^c$ with the symplectic form induced by the standard symplectic form whose  Hamiltonian vectorfield (i.e. the symplectic gradient) is defined by
\begin{equation}
X_H(\vw)\=K \big[\nabla H(\vw)+\omega M_\Gamma \vw\big].
\end{equation}
\end{rem}
The variational equation associated to the Hamiltonian system given in Equation \eqref{eq:equation-motion-rotating} is
\begin{equation}\label{eq:equation-motion-rotating-linear}
	M_\Gamma \dot \vxi(t)= K \big[D^2 H(\vw(t))+\omega M\big] \vxi(t)\big), \qquad t \in \left[0, \dfrac{2\pi}{|\omega|}\right].
\end{equation}
In particular if $\vw(t)= e^{\omega Kt} \vxi$ is a relative equilibrium solution at the central configuration $\vxi$, and the admissible variations belongs to the tangent along the fibers of the principal $S^1$-bundle, then 
\[
D^2 H(\vw(t))= D^2 H(\vxi).
\] 
In fact, since $H$ is invariant under rotation, it follows that  $H\big(\vz(t)\big)=H\big(\vxi\big)$. Now, by  differentiating twice this last equality for  $\vw(t)=e^{\omega K t}\vxi$ (here the admissible variations are of the form $e^{\omega K t} \vu$ for $\vu \in T \S_N^\vc$), we get  
\begin{equation}\label{eq:or-li-eq}
e^{-\omega Kt}\,D^{2}H\big(\vw(t)\big)\,e^{\omega Kt}=D^{2}H\left(\vz_{0}\right).
\end{equation}
Inserting the expression given in Equation \eqref{eq:or-li-eq} into Equation \eqref{eq:equation-motion-rotating} and setting  $\veta=e^{\omega Kt}\vxi$, we get 
\begin{equation}\label{eq:li-eq}
\dot{\veta}(t)=K\big[M_\Gamma^{-1}D^{2}H(\vxi)+\omega \Id\big]\veta(t).
\end{equation}
Following Roberts in \cite{Rob13} we introduce the following definition. 
\begin{defn}\label{def:stability-matrix}
The   matrix 
\begin{equation}\label{eq:B}
B(\vxi)\=K  A_\Gamma(\vxi) \quad \textrm{ where } \quad  A_\Gamma(\vxi)\=\big[M_\Gamma^{-1}D^{2}H(\vxi)+\omega \Id\big]
\end{equation}
 is termed the {\em stability matrix of the relative equilibrium $\vz$ generated by $\vxi$.\/}
\end{defn}

\begin{note}
When no confusion may occur, in shorthand notation we denote by $B$ (resp. $ A_\Gamma$) the matrices  $B(\vxi)$ (resp. $ A(\vxi)$). 

%Moreover, we denote throughout the matrix $M_\Gamma^{-1}D^{2}H(\vxi)$ by the symbol $M_\Gamma^{-1}D^{2}H(\vxi)(\vxi)$ or simply by $M_\Gamma^{-1}D^{2}H(\vxi)$. 
\end{note}
\begin{defn}\label{def:N-symmetry}
Let $(\R^k, N)$ be a (maybe indefinite) non-degenerate scalar product space on   the real vector space $\R^k$ and let $G \in \Mat(k,\R)$. The matrix $G$ is termed a {\em $N$-symmetric matrix\/}  if
 \[
 NG= \trasp{G}N,
 \]
 where $\trasp{\cdot}$ denotes the transpose with the respect to the Euclidean product. 
 The matrix $R \in \Sp(2n,\Omega)$ is termed {\em $N$-Hamiltonian\/}, if 
 \[
 \trasp{R}P N+ NP R=0
 \]
where $P$ represents the symplectic form $\Omega$ with respect to the Euclidean product. 
\end{defn}
\begin{rem}
It is worth noticing that if $N=\Id$ then 	the definitions of $N$-symmetric (resp. $N$-Hamiltonian) matrix, reduces to the standard definition of symmetric (resp. Hamiltonian) matrix with respect to the Euclidean scalar product (resp. canonical symplectic structure).
\end{rem}
By a direct calculation follows that  $M_\Gamma^{-1}D^{2}H(\vxi)(\vxi)$ is a $M_\Gamma$-symmetric matrix (whatever the sign of each circulation is) and $B(\vxi)$ is $M_\Gamma$-Hamiltonian, meaning that 
\[
\trasp{B}(\vxi)KM_\Gamma + M_\Gamma KB(\vxi)=0.
\]
This last claim directly follows by Equation \eqref{eq:impRob}.  Otherwise stated, the matrix $B(\vxi)$ is Hamiltonian with respect  to the {\em vortex symplectic form \/} defined by 
\begin{equation}
\omega_\Gamma(\cdot,\cdot)=\langle K\cdot, \cdot \rangle_\Gamma= \langle M_\Gamma\, K \cdot, \cdot \rangle.	
\end{equation}
\begin{rem}\label{rem:interessante}
We pause the exposition by  introducing the following remark that explain why the mixed sign circulations case is really completely different from the constant sign circulations.
It is well-known that the product of two symmetric matrices is symmetric as soon as the two matrices commute. Thus,  in general, the matrix $M_\Gamma^{-1}D^{2}H(\vxi)(\vxi)$ whatever $\vxi$ is, and no matter how the signs of $\Gamma_i$ arem  not be symmetric. Clearly, if all circulations are equal then $M_\Gamma$ is just a multiple of the identity and, of course, $M_\Gamma^{-1}D^{2}H(\vxi)$ is symmetric. However, if the circulations strengths $\Gamma_i$ are all positive, then the matrix $M_\Gamma$ is positive definite and in particular its spectrum is real. For mixed signs circulation strengths, however, $M_\Gamma$-is (nondegenerate) but indefinite and its spectrum \cite[Theorem 5.1.1, pag.74]{GLR05} is not necessarily real, anymore and this fact is responsible among others of some technicalities as well as a deep change in the dynamics of the problem.  
\end{rem}
We conclude this section by showing a nice and important block matrix structure of the Hessian matrix $D^{2}H(\vxi)$. This property comes from item (iii) in Lemma \ref{thm:property-of-H}. 
Let $\vxi$ be a central configuration; thus $\vxi=(\vxi_{1},\vxi_{2},\ldots,\vxi_{n})\in\R^{2n}$
and let $\vxi_{ij}=(\vxi_{i}-\vxi_{j})/r_{ij}$. A direct computation shows that
\begin{equation}
D^{2}H(\vxi)\;=\;\begin{bmatrix}A_{11} & A_{12} & \cdots & A_{1n}\\
\vdots &  &  & \vdots\\
A_{n1} & A_{n2} & \cdots & A_{nn}
\end{bmatrix}\label{eq:D2H}
\end{equation}
 where $A_{ij}$ is the $2\times2$ symmetric matrix
 \begin{equation}\label{eq:Aij}
\begin{cases} 
A_{ij}\=\dfrac{\Gamma_{i}\Gamma_{j}}{r_{ij}^{2}}[I-2\vxi_{ij}\trasp{\vxi_{ij}}] & \textrm{ if }i\neq j\\
 A_{ii}\=-\sum_{j\neq i}A_{ij} &\textrm{ otherwise.}
\end{cases}
\end{equation}
Note that $A_{ij}=A_{ji}$ and, for $i\neq j$ 
\[
A_{ij}=\dfrac{\Gamma_{i}\Gamma_{j}}{r_{ij}^{4}}\begin{bmatrix}(y_{i}-y_{j})^{2}-(x_{i}-x_{j})^{2} & -2(x_{i}-x_{j})(y_{i}-y_{j})\\[0.1in]
-2(x_{i}-x_{j})(y_{i}-y_{j}) & (x_{i}-x_{j})^{2}-(y_{i}-y_{j})^{2}
\end{bmatrix}
\]
where $\vxi=(x,y)$. The fact that $J$ commutes with each $A_{ij}$ gives another proof
of the fact that $D^{2}H(\vxi)$ and $K$ anti-commute.

\begin{lem}\label{thm:invariant-subspaces}
	The following facts hold:
	\begin{enumerate}
	\item $\vs=[1,0,1,0,\ldots,1,0],\ K\vs\in\ker D^{2}H(\vxi)$ ,\\
	\item For every relative equilibrium $\vz$ generated by the central configuration $\vxi$, we have that  $K\vxi\in\ker B$.
	\end{enumerate}
\end{lem}
\begin{proof}
	From the conservation of the center of vorticity, by using  Equation \eqref{eq:Aij} and by a straightforward  calculation we get that  
	\begin{equation*}
	 \sum_{i=1}^{n}A_{ji}\begin{pmatrix}1\\
	0
	\end{pmatrix}=\sum_{i\neq j}A_{ji}\begin{pmatrix}1\\
	0
	\end{pmatrix}+A_{jj}\begin{pmatrix}1\\
	0
	\end{pmatrix}
	=  \sum_{i\neq j}A_{ji}\begin{pmatrix}1\\
	0
	\end{pmatrix}-\sum_{i\neq j}A_{ji}\begin{pmatrix}1\\
	0
	\end{pmatrix}=0.
	\end{equation*}
	This implies that 
	\begin{equation*}
	D^{2}H(\vxi)\begin{pmatrix}1\\
	0\\
	1\\
	0\\
	\vdots\\
	1\\
	0
	\end{pmatrix} =\begin{pmatrix}A_{11} & A_{12} & \cdots & A_{1n}\\
	\vdots &  &  & \vdots\\
	A_{n1} & A_{n2} & \cdots & A_{nn}
	\end{pmatrix}\begin{pmatrix}1\\
	0\\
	1\\
	0\\
	\vdots\\
	1\\
	0
	\end{pmatrix}
	=\begin{pmatrix}\sum_{i=1}^{n}A_{1i}\begin{pmatrix}1\\
	0
	\end{pmatrix}\\
	\vdots\\
	\sum_{i=1}^{n}A_{ni}\begin{pmatrix}1\\
	0
	\end{pmatrix}
	\end{pmatrix}=0.
	\end{equation*}
	Thus $\vs \in \ker D^2H(\vxi)$. \\
	Furthermore, by the last claim in Lemma \ref{thm:property-of-H}, we have $D^2H(\vxi)K\vs=-KD^2H(\vxi)\vs=0$ which implies that also $K\vs \in \ker D^2H(\vxi)$.		 
	
The  proof of the second claim directly follows by 	Equation \eqref{eq:impRob} together with property (iii) Lemma \ref{thm:property-of-H}. In fact, by a direct computation we get
	\begin{multline}
B\vxi= K[M_\Gamma^{-1}D^2 H(\vxi) \vxi + \omega\, \vxi]= K[\omega \vxi + \omega \vxi]= 2\, \omega K\vxi\\
BK\vxi= 	 K[M_\Gamma^{-1}D^2 H(\vxi) K\vxi + \omega\,K \vxi]= 
K[-K M_\Gamma^{-1}D^2 H(\vxi) \vxi + \omega\,K \vxi]\\= K[-\omega\,K \vxi+\omega\,K \vxi]=0.
\end{multline}
This concludes the proof. 
\end{proof}
By the first claim in Lemma \ref{thm:invariant-subspaces} it then follows that the restriction of the stability matrix to  the spectrum of $B$ is precisely $\{\pm \omega i\}$.
For a given relative equilibrium $\vz$ with corresponding central configuration $\vxi$, let $W=\textrm{span}\{\vxi,K\vxi\}$. As already proved in Lemma \ref{thm:invariant-subspaces} this is an invariant subspace for $B$ and the restriction of $B$ to $W$ is given by 
\[
B\Big\vert_{W}=\begin{bmatrix}0 & 0\\
2\omega & 0
\end{bmatrix}
\]
\begin{note}
In what follows, we denote by $W^{\perp}\subset V^N$  the $M_\Gamma$-orthogonal complement of $W$, that is, 
\[
W^{\perp}=\Set{\vw\in\R^{2n}:\Gscal{\vw, \vv}=\trasp{\vw}M_\Gamma\vv=0, \forall \vv\in W}.
\]
\end{note}
\begin{lem}\label{lemma:Vperp}
The vector space $W^{\perp}$ has
dimension $2n-2$ and is invariant under $B$. If $L\neq0$, then
$W\cap W^{\perp}=\{0\}$. 
\end{lem}
\begin{proof}
For the proof of this result, we refer the interested reader to \cite[Lemma 2.6]{Rob13}. 
\end{proof}
As long as $L\neq0$, Lemma $\ref{lemma:Vperp}$ allows us to define
the linear stability with respect to the $M_\Gamma$-orthogonal complement of the subspace $W$. Thus, a relative equilibrium is spectrally (resp. linearly) stable if the restriction of the matrix $B$ onto $W^\perp$ is spectrally (resp. linearly) stable according to Definition \ref{def:stability}. However 
instead of working on a reduced phase space (eliminating the rotational symmetry), it is more convenient to work on the full phase space we investigate the linear stability properties of the orbits  in the full space. However, in this case, by the invariance properties of $H$, it readily follows that $4$ is the minimal possible nullity (or kernel dimension) of the corresponding linear differential operator.

%%%%%%%%%%%%%%%%%%%%%%%%%%%%%%%%%%%%%%%%%%%%%%%%%%%%%%%%%%%%
%%%%%%%
%%%%%%%
%%%%%%%
%%%%%%%%%%%%%%%%%%%%%%%%%%%%%%%%%%%%%%%%%%%%%%%%%%%%%%%%%%%%

\section{Spectral properties and canonical forms  of the stability matrix}\label{sec:Some Linear Algebraic preliminaries}
 
This section is devoted to collect several linear algebraic results on the spectrum of the stability matrix $B$  that we will need in order to prove our main result. In fact, as already observed,  the eigenvalues of $B$ determine the spectral and linear stability of the corresponding periodic solution.

%%%%%%%%%%%%%%%%%%%%%%%%%%%%%%%%%%%%%%%%%%%%%%%%%%%%%%%%%%%%
%%%%%%%
%%%%%%%
%%%%%%%
%%%%%%%%%%%%%%%%%%%%%%%%%%%%%%%%%%%%%%%%%%%%%%%%%%%%%%%%%%%%

\subsection{Spectral properties of the stability matrix}

The aim of this subsection is to study the relation intertwining the spectrum of the stability matrix $B$ and the spectrum of $M_\Gamma^{-1}D^{2}H(\vxi)$ (and hence of $\widehat A$). The first result, that we recall here for the sake of the reader, was proved by  Roberts in \cite{Rob13}.
\begin{lem}\label{thm:simple-eigenvalue block-1} The characteristic polynomials of $M_\Gamma^{-1}D^{2}H(\vxi)$ and $D^{2}H(\vxi) $ are even. 

Moreover, for each on of the above matrices, $\vv$ is an eigenvector with eigenvalue $\mu$ if and only if $K\vv$ is an eigenvector with eigenvalue $- \mu$.
\end{lem}
\begin{proof}
The proof of this result follows by a direct computation. (Cf. \cite[Lemma 2.4]{Rob13}, for further details).
\end{proof}

The following result relates the spectrum of $B$ with the spectrum of the matrix $M_\Gamma^{-1}D^{2}H(\vxi)$ and will be a key ingredient for the stability analysis. 
\begin{lem}\label{thm:spectrum-B-Hessian} 
Under the above notation, $\lambda \in \sigma (B)$ if and only if  $\mu_\pm\=\pm \sqrt{\lambda^2 + \omega^2}\in\sigma(M_\Gamma^{-1}D^{2}H(\vxi))$, where $\sqrt{\cdot}$ denotes the the square root of the maybe negative/complex number $\lambda^2 + \omega^2$. More precisely,  the following facts hold:
\begin{enumerate}
\item $\lambda \in \sigma(B) \cap i\R)$ iff $\mu \in \sigma(M_\Gamma^{-1}D^{2}H(\vxi)) \cap i\R$ or $\mu \in \sigma(M_\Gamma^{-1}D^{2}H(\vxi)) \cap \R$ and  $|\mu| \in \big[-|\omega|, |\omega|\big]$. 
\item $\lambda \in \sigma(B) \cap \R)$ iff $\mu \in \sigma(M_\Gamma^{-1}D^{2}H(\vxi))\cap \R$ and $|\mu| \geq  |\omega|$.
\item $\lambda \in \sigma(B) \cap \C\setminus\big(\R \cup i \R\big)$ iff $\mu \in \sigma(M_\Gamma^{-1}D^{2}H(\vxi)) \cap \C\setminus\big(\R \cup i \R\big)$.
\end{enumerate}
\end{lem}
\begin{proof}
Since $K^{2}=-\Id$ and $K M_\Gamma^{-1}D^{2}H(\vxi)=- M_\Gamma^{-1}D^{2}H(\vxi) K$ (as direct follows by applying Lemma \ref{thm:property-of-H}), then by a direct computation, we get: 
\begin{align}\label{eq:relation-of-stability-matrix}
B-\lambda^2 \Id&=(B-\lambda)(B+\lambda)=  \big[K(M_\Gamma^{-1}D^{2}H(\vxi)+\omega \Id)- \lambda \Id\big]\big[K(M_\Gamma^{-1}D^{2}H(\vxi)+\omega \Id)+\lambda \Id\big] \nonumber\\
&= K(M_\Gamma^{-1}D^{2}H(\vxi) + \omega \Id) K(M_\Gamma^{-1}D^{2}H(\vxi) + \omega \Id)-\lambda^2\Id= (KM_\Gamma^{-1}D^{2}H(\vxi) +K\omega \Id)^2 -\lambda^2 \Id \nonumber
\\&= M_\Gamma^{-1}D^{2}H(\vxi)^2 + KM_\Gamma^{-1}D^{2}H(\vxi) K\omega+ K\omega K M_\Gamma^{-1}D^{2}H(\vxi)-\lambda^2\Id-\omega^2\Id\\&= M_\Gamma^{-1}D^{2}H(\vxi)^2+ M_\Gamma^{-1}D^{2}H(\vxi) \omega -\omega M_\Gamma^{-1}D^{2}H(\vxi) -\lambda^2\Id-\omega^2\Id\\ &= M_\Gamma^{-1}D^{2}H(\vxi)^2 -(\lambda^2+\omega^2)\Id.
\end{align}
By the calculation performed in Equation \eqref{eq:relation-of-stability-matrix} we get that 
\[
0=\det(B-\lambda^2 \Id) \quad \textrm{  iff } \quad \det[M_\Gamma^{-1}D^{2}H(\vxi)^2-(\lambda^2+\omega^2)\Id]=\det[M_\Gamma^{-1}D^{2}H(\vxi)^2-\mu^2\Id]=0,
\] 
where $\mu^2\=\lambda^2 +\omega^2$. 
 
In order to prove the first claim, we start to observe that if  $\lambda \in \sigma(B) \cap i\R$ then $\lambda^2= \mu^2-\omega^2<0$. By this last inequality we can conclude that either $\mu^2 <0$ or $0<\mu^2\leq \omega^2$ or which is equivalent to state that $\mu \in \sigma(M_\Gamma^{-1}D^{2}H(\vxi)) \cap i\R$ or $\mu \in \sigma(M_\Gamma^{-1}D^{2}H(\vxi)) \cap \R$ and  $|\mu| \in \big[-|\omega|, |\omega|\big]$. Viceversa, if $\mu \in \sigma(M_\Gamma^{-1}D^{2}H(\vxi)) \cap i\R$, in particular  $\mu^2 <0$. Being $\mu^2= \lambda^2 +\omega^2<0$ and $\omega^2>0$ (being $\omega \in \R$) this implies that $\lambda^2<0$. Thus  $\lambda \in  i\R$ and by Equation \eqref{eq:relation-of-stability-matrix} $\lambda \in \sigma(M_\Gamma^{-1}D^{2}H(\vxi))$. This conclude the proof of the first item . The proof of items 2 and 3  can be proved by arguing precisely as above and we leave  the proof to the interested reader. 
\end{proof}

%%%%%%%%%%%%%%%%%%%%%%%%%%%%%%%%%%%%%%%%%%%%%%%%%%%%%%%%%%%%
%%%%%%%
%%%%%%%
%%%%%%%
%%%%%%%%%%%%%%%%%%%%%%%%%%%%%%%%%%%%%%%%%%%%%%%%%%%%%%%%%%%%

\subsection{Canonical forms and invariant splitting of the phase space}

This subsection is devoted to study the relation between the invariant subspaces of $B$ (which are crucial for reducing the operator $B$ and the generalized eigenspaces of $ A_\Gamma$).  Lemma \ref{thm:simple-eigenvalue block} that we state below for the sake of the reader, was proved  in \cite[Lemma 2.5]{Rob13}.
\begin{lem}\label{thm:simple-eigenvalue block}
 Let $p(\lambda)\= \det(B-\lambda \Id)$ be the characteristic polynomial of the stability matrix $B$.
\begin{itemize} 
\item[{\bf (a)}] Suppose that $\vv$ is a real eigenvector of $M_\Gamma^{-1}D^{2}H(\vxi)$ corresponding to the eigenvalue $\mu$. Then the span of the two vectors  $\{ \vv, K\vv \}$ is a real invariant subspace of $B$ and the restriction of $B$ to this subspace is given by 
\begin{equation}\label{eq:mat1}
\begin{bmatrix} 0 &\mu - \omega \\ \mu + \omega & 0 \\ 
\end{bmatrix}.  
\end{equation} 
Consequently, $p(\lambda)$ has a quadratic factor of the form $\lambda^2 + \omega^2 - \mu^2$.
\item[{\bf (b)}] 
Suppose that $\vv = \vv_1 + i \, \vv_2$ is a complex eigenvector of $M_\Gamma^{-1}D^{2}H(\vxi)$ corresponding to the complex eigenvalue $\mu = \alpha + i \, \beta$. Then  the span of the four vectors $\{ \vv_1,\vv_2, K\vv_1, K\vv_2 \}$ is a real invariant subspace of $B$ and the restriction of $B$ to this subspace is given by 
 \begin{equation}\label{eq:mat2} 
 \begin{bmatrix} 0 & 0 & \alpha - \omega & \beta \\ 0 & 0 & -\beta & \alpha - \omega \\ \alpha + \omega & \beta & 0 & 0 \\ -\beta &\alpha + \omega & 0 & 0 \end{bmatrix}.
 \end{equation} 
 Consequently, $p(\lambda)$ has a quartic factor of the form $(\lambda^2 + \omega^2 - \mu^2)(\lambda^2 + \omega^2 - \overline{\mu}^{\, 2})$. 
 \end{itemize}
\end{lem}
\begin{proof}
The proof of this result follows by a direct computation by using Lemma \ref{thm:simple-eigenvalue block-1}. (Cf. \cite[Lemma 2.5]{Rob13}).
\end{proof}
In the case of mixed signs circulations, the matrix $M_\Gamma^{-1}D^{2}H(\vxi)$ is $M_\Gamma$-symmetric with respect to an indefinite scalar product and this, among others, in particular implies that the spectrum is not real and $M_\Gamma^{-1}D^{2}H(\vxi)$ and hence $ A_\Gamma$ are not semi-simple. In order  to decompose the full space into $B$ invariant subspaces it is then crucial to understand in which manner  Lemma  \ref{thm:simple-eigenvalue block} can be carried over in this more general situation we are dealing with.  This is essentially the content of Lemma   \ref{thm:conjugated-jordan-block} and Proposition \ref{thm:general eigenvalue block}, below.
\begin{lem}\label{thm:conjugated-jordan-block}
Let $\{\vv_{i}\}_{i=0}^{l}$ be a Jordan
chain of $A_\Gamma$ with eigenvalue $\nu$, namely 
\begin{equation}
\begin{cases}
A_\Gamma\,\vv_{i+1}=\nu \,\vv_{i+1}+\vv_{i}\\
 \vv_{0}=0.
\end{cases}
\end{equation}
Then, the set  $\{K\vv_{i}\}_{i=0}^{l}$ is a Jordan
chain of $A_\Gamma$ with eigenvalue $2\omega-\nu$; thus 
\begin{equation}
\begin{cases}
A_\Gamma K\vv_{i+1}=(2\omega-\nu)\,K\vv_{i+1}-K\vv_{i}
\\
 \vv_{0}=0.
\end{cases}
\end{equation}
\end{lem}
\begin{proof}
Since $KM_\Gamma^{-1}D^{2}H(\vxi)=-M_\Gamma^{-1}D^{2}H(\vxi) K$ (as directly follows by the third item in Lemma \ref{thm:property-of-H}), then  by a direct computation, we get  
\begin{align*}
\big[M_\Gamma^{-1}D^{2}H(\vxi)+\omega \Id\big]K\vv_{i+1}= & K(-M_\Gamma^{-1}D^{2}H(\vxi)+\omega I)\vv_{i+1}\\
= & K(-M_\Gamma^{-1}D^{2}H(\vxi)-\omega \Id+2\omega I)\vv_{i+1}\\
= & -KA_\Gamma\vv_{i+1}+2\omega K\vv_{i+1}\\
= & (2\omega-\nu)K\vv_{i+1}-K\vv_{i}.
\end{align*}
This concludes the proof. 
\end{proof}
Lemma 	\ref{thm:conjugated-jordan-block} provides a constructive way to reduce the operator $B$ (by decomposing the whole space into $B$-invariant subspaces).
\begin{prop}\label{thm:general eigenvalue block}
Let $\{\vv_{i}\}_{i=0}^{l}$ be a Jordan
chain of $ A_\Gamma$ with eigenvalue $\nu$. Then the span of $\Set{\vv_{i},K\vv_{i}}_{i=1}^{l}$ is an invariant space for  $B$.
\end{prop}
\begin{proof}
Since $M_\Gamma^{-1}D^{2}H(\vxi)$ and $K$ anti-commutes, namely  $KM_\Gamma^{-1}D^{2}H(\vxi)=-M_\Gamma^{-1}D^{2}H(\vxi) K$, by a direct computation of  $B\vv_{i+1}$ and
$BK\vv_{i+1}$ and by taking advantage of Lemma 	\ref{thm:conjugated-jordan-block},  we immediately get  
\[
\begin{cases}
	B\vv_{i+1}=\nu\,  K\vv_{i+1}+K\vv_{i}\\
	BK\vv_{i+1}=(\nu -2\omega)\vv_{i+1}+\vv_{i}.
\end{cases}
\]
These two last equalities imply that the subspace generated by  $\Set{\vv_{i},K\vv_{i}}_{i=1}^{l}$ is an invariant space of $B$. This concludes the proof. 
\end{proof}
\begin{note}
We introduce the following 
\begin{itemize}
\item[] \begin{equation}
\varGamma_p(a,s)\=\begin{bmatrix} 
a & s & \cdots & 0 & 0\\
0 & a& \cdots & 0 & 0\\
\vdots & \vdots & \ddots & \vdots & \vdots\\
0 & 0 & \cdots & a & s\\
0 & 0 & \cdots & 0 & a
\end{bmatrix},
\end{equation}where $p$ denotes by the order of this matrix.
 \item[] If $\nu\in\sigma(A_\Gamma)$ we denote by $E_\nu$ the (real) generalized spectral space  corresponding to the eigenvalue $\nu$.
 \end{itemize}
\end{note}
Directly from Proposition \ref{thm:general eigenvalue block} and by using notation above,  we get that the restriction of $B$ onto the subspace $\widetilde E_\nu\=E_\nu \oplus E_{\nu- 2\omega}$ can be represented by the following $2l\times 2l$ matrix that in block matrix form can be written as follows
\begin{equation}
B\Big\vert_{\widetilde E_\nu}= \begin{bmatrix}
 0_l&  \Gamma_l(\nu,1)\\ 
\Gamma_l(\nu -2\omega,-1) & 0_l
\end{bmatrix}.
\end{equation}

%%%%%%%%%%%%%%%%%%%%%%%
%%%%
%%%%
%%%%%%%%%%%%%%%%%%%%%%%

{{\subsection{$A_\Gamma$-invariant $M_\Gamma$-orthogonal decomposition }}   }

Let us start to introduce the following symmetric matrix 
\begin{equation}\label{def:widehat-A}
\widehat A_\Gamma\=M_\Gamma A_\Gamma=D^{2}H(\vxi)+\omega M_\Gamma.
\end{equation}
 The rest of this section a bit technical in its own and the basic idea behind is to establish the behavior of the restriction of  $\widehat A$ onto some subspaces constructed through the spectral subspaces (maybe generalized spectral subspaces) of 
		$A_\Gamma$ which, as consequence of Corollary \ref{eq:orthogonal}, are $\widehat A_\Gamma$-orthogonal.

The next result provide a sufficient condition in order the generalized spectral subspaces relative to different and not conjugated eigenvalues to be $M_\Gamma$-orthogonal.
\begin{lem}\label{thm:orthogonal-1}
Suppose that $\nu_{1},\nu_{2}\in\sigma(A_\Gamma)$. If $\nu_{1}\neq\overline{\nu_{2}}$, then  for every $\vv_{1}\in E_{\nu_{1}}$ and $\vv_{2}\in E_{\nu_{2}}$ we have 
\[
\left\langle \vv_1, \vv_2\right\rangle_{M_\Gamma}=\left\langle M_\Gamma \vv_{1},\vv_{2}\right\rangle =0.
\]
\end{lem}
\begin{proof}
We split the proof into two steps.\\ 
\underline{First step}. We assume that   $\vv_{1},\vv_{2}$ are eigenvectors relative to the eigenvalues $\nu_1$ and $\nu_2$ respectively; thus, we have  
$A_\Gamma\vv_{i}=\nu_{i}\vv_{i},\ i=1,2$. So, we have  
\begin{multline*}
\left\langle \widehat A_\Gamma\vv_{1},\vv_{2}\right\rangle =\left\langle M_\Gamma A_\Gamma\vv_{1},\vv_{2}\right\rangle =  \left\langle M_\Gamma\nu_{1}\vv_{1},\vv_{2}\right\rangle 
=  \nu_{1}\left\langle M_\Gamma\vv_{1},\vv_{2}\right\rangle \textrm{ and }\\
\left\langle \widehat{A}\vv_{1},\vv_{2}\right\rangle =\left\langle \vv_{1},M_\Gamma A_\Gamma \vv_{2}\right\rangle  
= \left\langle \vv_{1},\nu_2 M_\Gamma \vv_{2}\right\rangle
=  \overline \nu_2 \left\langle \vv_1,M_\Gamma\vv_{2}\right\rangle,
\end{multline*}
since $\nu_{1}\neq\overline{\nu_{2}},$ we get desired result.\\
\underline{Second step}. We assume that   $\vv_{1},\vv_{2}$ are  generalized eigenvectors and we consider  the Jordan chains
\begin{align*}
A_\Gamma\vv_{1}^{i+1}= & \nu_{1}\vv_{1}^{i+1}+\vv_{1}^{i},\qquad  \forall \, i \in \{0,1, \dots, p\} \\
A_\Gamma\vv_{2}^{j+1}= & \nu_{2}\vv_{2}^{j+1}+\vv_{2}^{j},\qquad \forall \, j \in \{0,1,\dots, q\}
\end{align*}
where $\vv_{1}^{0}=\vv_{2}^{0}=0$. By arguing as above, we get that 
\begin{multline}\label{eq:organal-1}
\left\langle M_\Gamma A_\Gamma \vv_{1}^{i+1},\vv_{2}^{j+1}\right\rangle =  \left\langle M_\Gamma\nu_{1}\vv_{1}^{i+1}+M_\Gamma\vv_{1}^{i},\vv_{2}^{j+1}\right\rangle  \\
=  \nu_{1}\left\langle M_\Gamma\vv_{1}^{i+1},\vv_{2}^{j+1}\right\rangle +\left\langle M_\Gamma\vv_{1}^{i},\vv_{2}^{j+1}\right\rangle. \textrm{ Moreover }\\ 
  \left\langle\vv_{1}^{i+1}, M_\Gamma A_\Gamma\vv_{2}^{j+1}\right\rangle 
=  \overline{\nu}_{2}\left\langle M_\Gamma\vv_{1}^{i+1},\vv_{2}^{j+1}\right\rangle +\left\langle M_\Gamma\vv_{1}^{i},\vv_{2}^{j+1}\right\rangle.
\end{multline}
So, taking the difference between the two equalities in Equation \eqref{eq:organal-1}, we get 
\begin{multline}\label{eq:orginal-2}
0=\left\langle M_\Gamma A_\Gamma \vv_{1}^{i+1},\vv_{2}^{j+1}\right\rangle-\left\langle\vv_{1}^{i+1}, M_\Gamma A_\Gamma\vv_{2}^{j+1}\right\rangle \\= (\nu_{1}-\overline{\nu}_{2})\left\langle M_\Gamma \vv_{1}^{i+1},\vv_{2}^{j+1}\right\rangle +\left\langle M_\Gamma \vv_{1}^{i},\vv_{2}^{j+1}\right\rangle -\left\langle M_\Gamma \vv_{1}^{i},\vv_{2}^{j+1}\right\rangle\\= (\nu_{1}-\overline{\nu}_{2})\left\langle M_\Gamma \vv_{1}^{i+1},\vv_{2}^{j+1}\right\rangle
\end{multline}
where the first equality follows by the fact that $A_\Gamma$ is $M_\Gamma$-symmetric and $\trasp{M_\Gamma}=M_\Gamma$.
Since $\nu_1 -\overline \nu_2 \neq 0$, it follows that  
\begin{equation}\label{eq:orginal-3}
\left\langle M_\Gamma \vv_{1}^{i+1},\vv_{2}^{j+1}\right\rangle =0 \quad \forall \, i \in \{0,1, \dots, p\} \textrm{ and } \forall \, j \in \{0,1, \dots, q\}.
\end{equation}
To conclude the proof we argue by induction. 
Let $i+j=k$. So  Equation \eqref{eq:orginal-3} is trivially true for$k=0$. Now, we suppose Equation \eqref{eq:orginal-3} holds true for $i+j=k\leq l$ and we want to prove that it is true for  $i+j=k=l+1$. Now,  by taking into account  Equation \eqref{eq:orginal-2} and being   $\nu_{1}\neq\overline{\nu}_{2}$, it readily follows that Equation  \eqref{eq:orginal-3}
holds true. 
This concludes the proof.  
\end{proof}
In particular generalized eigenspaces relative to different and not conjugated eigenvalues are $\hat{A}_\Gamma$ orthogonal.
\begin{cor} \label{eq:orthogonal}
Suppose that $\nu_{1},\nu_{2}\in\sigma(A_\Gamma)$. If $\nu_{1}\neq\overline{\nu_{2}}, $ then we have $\left\langle \widehat A_\Gamma \vv_{1},\vv_{2}\right\rangle =0$   for all $\vv_{1}\in E_{\nu_{1}}$and $\vv_{2}\in E_{\nu_{2}}$. 
\end{cor}
For $\nu\in\sigma(A_\Gamma)$, we set
\[
I_{\nu}\=\begin{cases}
E_{\nu}\oplus E_{\bar{\nu}}& \textrm{ if }\nu\notin \R \\
E_{\nu} & \textrm{ otherwise}.
\end{cases}
\]

\begin{note}
We introduce the following notation. Given any subspace $X\subset \C^{2n},$ we denote by $ \iMor(\widehat A_\Gamma|_{X})$  (resp. $\coiMor(\widehat A_\Gamma|_{X})$), the dimension of the  maximal negative (resp.  positive) spectral subspace of the restriction of the quadratic form $\left\langle \widehat A_\Gamma\cdot,\cdot\right\rangle $ onto $X$. 
\end{note}
By the previous discussion, we can decompose the $\C^{2n}= \R^{2n}\otimes_\C \C$ into $A_\Gamma$-{invariant } , $M_\Gamma$-orthogonal subspaces; thus we have $\C^{2n}=I_{\nu_{1}}\oplus\cdots\oplus I_{\nu_{l}},$ where
$\nu_{i}$ are all distinct eigenvalues of $A_\Gamma$ with $\Im\nu_{i}\geq 0$. 
%{\color{red}  \begin{lem}\label{thm:invariant-subspaces-A-widehat-A}\marginpar{ after discussion with Professor Hu, we think  Lemma \ref{thm:invariant-subspaces-A-widehat-A} is wrong, note that $I_{2\omega}=\textrm{span}\{K\xi \}$ and $\hat{A}_\Gamma K\xi=MK\xi, $ this implies that   $I_{2\omega} $ is not $\widehat A_\Gamma$-invariant.  Therefore the proof of our mainly theorem only need  the fact that generalized eigenspaces relative to different and not conjugated eigenvalues are$ A_\Gamma$-invariant, $\hat{A}_\Gamma$ orthogonal and $M_\Gamma$ orthognal.  }
%	Under the above notation, the subspace $I_\nu$ is $\widehat A_\Gamma$-invariant.
%\end{lem}
%\begin{proof}
%In shorthand notation, we denote by $A^\nu_\Gamma$ (resp. $M_\Gamma^\nu$) the restriction of $A_\Gamma$ (resp. $M_\Gamma)$ to $I_\nu$. By assumption, $A^\nu_\Gamma(I_\nu) \subseteq I_\nu$. Furthermore 
%\[
%\widehat A_\Gamma^\nu(I_\nu) = M_\Gamma^\nu \big(A^\nu_\Gamma(I_\nu)\big)= A^\nu_\Gamma\big( M_\Gamma^\nu (I_\nu)\big)=  A^\nu_\Gamma(I_\nu)\subseteq I_\nu.
%\] 	
%Then $\widehat A_\Gamma (I_\nu) \subseteq I_\nu$ and hence $I_\nu$ is $A_\Gamma$-invariant. This concludes the proof. 
%\end{proof}}

\begin{lem}\label{thm:non-degenerate}
Let $\nu\in\sigma(A_\Gamma)$
and assume $\nu\neq0$. Then the restriction 
\[
\left\langle \widehat A_\Gamma \cdot,\cdot\right\rangle |_{I_{\nu}}
%= \left\langle  A_\Gamma \cdot,\cdot\right\rangle_{M_\Gamma} |_{I_{\nu}}
\]
is non-degenerate.
\end{lem}
\begin{proof}
We start to observe that
%being $\nu \neq 0$ then the restriction of $A_\Gamma$ onto $I_\nu$ is clearly non-degenerate. Moreover since
  $\ker\widehat A_\Gamma\subseteq E_{0}=\ker A_\Gamma$ \footnote{Actually if $M_\Gamma$    is positive definite it can be proved that also the converse inclusion holds}. Now, arguing by contradiction,  we assume that  $\left\langle \widehat A_\Gamma \cdot,\cdot\right\rangle |_{I_{\nu}}$ is
degenerate for some $\nu \neq 0$. Thus  there exists $ \vu\in I_{\nu}$ and $\vu 0\neq$, such that 
\[
\left\langle\widehat A_\Gamma\vu,\vv\right\rangle =0\qquad for all \vv\in I_{\nu}.
\]
 Since $\C^{2n}$is the direct sum of all different
$I_{\nu'}$, where $\nu'\in\sigma(A_\Gamma)$, this implies (by invoking by Lemma \ref{thm:orthogonal-1})
that $\left\langle \widehat A_\Gamma\vu,\vv\right\rangle =0$ for
all $\vv\in \C^{2n}$. So $\vu \in \ker \widehat A_\Gamma$ and in hence $\vu \in \ker A_\Gamma= E_0$.  Thus $\nu=0$ which is  a contradiction. This concludes the proof.
\end{proof}
The next result shed some light on the relation between the dimension of $I_\nu$ and the Morse index .

\begin{lem} \label{thm:positive =negative}
Let $\nu \in \sigma(A_\Gamma)$,  $\Im \nu>0$ and let $m_{\nu} \in \N$ be its algebraic multiplicity. Then 
\[
\iMor\left[\widehat A_\Gamma|_{I_{\nu}}\right]=m_{\nu}.
\]
\end{lem}
\begin{proof}
By using  Lemma \ref{thm:orthogonal-1}, the quadratic form $\left\langle\widehat A_\Gamma\cdot,\cdot\right\rangle$
on $I_{\nu}=E_{\nu}\oplus E_{\bar{\nu}}$ can be represented in the
block matrix form by $\begin{bmatrix}0 & Y\\
\trasp{Y} & 0
\end{bmatrix}$ for some $Y \in \Mat(\nu; \C)$. Moreover, by Lemma \ref{thm:non-degenerate} we infer that  $Y$ is non-degenerate and by  this fact the conclusion readily follows.
\end{proof}
By taking into account Corollary \ref{eq:orthogonal} as well as the additivity property of the inertia indices of $\widehat A$  with respect to the direct sum decomposition of the space into $\widehat{A}_\Gamma$-orthogonal subspaces,  it follows that 
\begin{equation}\label{eq:morse}
\iMor(\widehat A_\Gamma)=\sum_{i=1}^l\iMor \big(\widehat A_\Gamma^{\nu_i} \big)
\end{equation}
where we set  $\widehat A_\Gamma^{\nu_i}\=\big[D^{2}H(\vxi)+\omega M_\Gamma\big]|_{I_{\nu_i}}$.

Let 
$\{\vv_{i}\}_{i=0}^{l}$ be a 
Jordan chain for for the generalized eigenspace $E_\nu$; thus 
\begin{equation}
	\begin{cases}
		A_\Gamma\vv_{i+1}=\nu \vv_{i+1}+\vv_{i}\\
		\vv_{0}=0.
	\end{cases}
\end{equation}
 By invoking  Lemma \ref{thm:conjugated-jordan-block}, 
$\{K\vv_{i}\}_{i=0}^{l}$ is a Jordan chain for the generalized eigenspace $E_{2\omega-\nu}$ relative to 
$2\omega-\nu$. More explicitly, the restriction
of $A_\Gamma$ into the subspace generated by $\{K\vv_{i}\}_{i=0}^{l}$ is given by 
${\Gamma}_l(2\omega-\nu,-1)$.

The next two results, Lemma  \ref{thm:matrix-1} and Lemma \ref{thm:matrix-2} will be very useful later on for computing the inertia indices  of $\widehat A_\Gamma^\nu$ in terms of that  of $M_\Gamma^\nu$. In Lemma  \ref{thm:matrix-1} we investigate such a relation by restricting on a single Jordan block. In Lemma \ref{thm:matrix-2} we assume that there exists two different Jordan blocks corresponding to the same eigenvalue.
\begin{lem}\label{thm:matrix-1}
 Let $0\neq\nu\in\sigma(A_\Gamma)\cap \R$ and let $\{\vv_{i}\}_{i=0}^{l}$ be a  Jordan chain for for the generalized eigenspace $E_\nu$. Under the previous notation we get that  
 \[
 \left\langle M_\Gamma\vv_{i},\vv_{j+1}\right\rangle =\left\langle M_\Gamma \vv_{i+1},\vv_{j}\right\rangle \textrm{ and }
\left\langle M_\Gamma\vv_{i},\vv_{j}\right\rangle =0 \textrm{ for } 1\leq j\leq l-i \textrm{ and } i=1,\cdots,l-1.
\]
\end{lem}
\begin{proof}
By direct computation we infer that 
\begin{multline*}
\left\langle \widehat A_\Gamma\vv_{i+1},\vv_{j+1}\right\rangle =  
\left\langle M_\Gamma  A_\Gamma\vv_{i+1},\vv_{j+1}\right\rangle=  \left\langle M_\Gamma(\nu \vv_{i+1}+\vv_{i}),\vv_{j+1}\right\rangle \\
= \nu\left\langle M_\Gamma \vv_{i+1},\vv_{j+1}\right\rangle +\left\langle M_\Gamma \vv_{i},\vv_{j+1}\right\rangle.\textrm{ Moreover } \\
 \left\langle \widehat A_\Gamma\vv_{i+1},\vv_{j+1}\right\rangle =\left\langle \vv_{i+1},\widehat A_\Gamma\vv_{j+1}\right\rangle = \left\langle \vv_{i+1},M_\Gamma  A_\Gamma\vv_{j+1}\right\rangle \\
=  \nu\left\langle M_\Gamma\vv_{i+1},\vv_{j+1}\right\rangle +\left\langle M_\Gamma\vv_{i+1},\vv_{j}\right\rangle.
\end{multline*}
By taking the difference of the first and last members in the previous equations, we get   $\left\langle M_\Gamma\vv_{i},\vv_{j+1}\right\rangle =\left\langle M_\Gamma\vv_{i+1},\vv_{j}\right\rangle .$
Being $\vv_{0}=0$, we infer  also  that $\left\langle M_\Gamma\vv_{i},\vv_{j}\right\rangle =0$
for every $1\leq j\leq l-i$ and $i=1,\cdots,l-1.$ This concludes the proof.
\end{proof}
\begin{lem}\label{thm:matrix-2} Let $0\neq\nu\in\sigma(A_\Gamma)\cap \R$ and we assume that $\{\vv_{i}\}_{i=0}^{p}$ and $\{\vw_{j}\}_{j=0}^{q}$ are two Jordan chains for the generalized eigenspaces relative to the same eigenvalue $\nu$ and such that $\vv_{0}=\vw_{0}=0$; furthermore we assume that $p\leq q$.   Then we have
$\left\langle M_\Gamma \vv_{i},\vw_{j+1}\right\rangle =\left\langle M_\Gamma\vv_{i+1},\vw_{j}\right\rangle $
and $\left\langle M_\Gamma\vv_{i},\vw_{j}\right\rangle =0\ for\ 1\leq i+j\leq q$.
\end{lem}
\begin{proof}
By a direct computation we get 
%gives that
%\begin{multline}
%\left\langle \widehat A_\Gamma\vv_{i+1},\vw_{j+1}\right\rangle = \left\langle M_\Gamma(\nu \vv_{i+1}+\vv_{i}),\vw_{j+1}\right\rangle \\
%=  \nu\left\langle M_\Gamma\vv_{i+1},\vw_{j+1}\right\rangle +\left\langle M_\Gamma\vv_{i},\vw_{j+1}\right\rangle \textrm{ and }  \\
%\left\langle \vv_{i+1},\widehat A_\Gamma \vw_{j+1}\right\rangle = \left\langle \vv_{i+1},M_\Gamma A_\Gamma\vw_{j+1}\right\rangle 
%=  \nu\left\langle M_\Gamma \vv_{i+1},\vw_{j+1}\right\rangle +\left\langle M_\Gamma\vv_{i+1},\vw_{j}\right\rangle ,
%\end{multline}
%this implies 
$\left\langle M_\Gamma\vv_{i},\vw_{j+1}\right\rangle =\left\langle M_\Gamma\vv_{i+1},\vw_{j}\right\rangle .$
Since $\vv_{0}=\vw_{0}=0$ and being $p\leq q$,  then we have $\left\langle M_\Gamma\vv_{i},\vw_{j}\right\rangle =0$  for $1\leq i+j\leq q$. This concludes the proof.
\end{proof}

%%%%%%%%%%%%%%%%%%%%%%%%%%%%%%%%%%%%%%%%%%%%%%%%%%%%%%%%%%%%
%%%%%%%
%%%%%%%
%%%%%%%
%%%%%%%%%%%%%%%%%%%%%%%%%%%%%%%%%%%%%%%%%%%%%%%%%%%%%%%%%%%%

\section{Proof of main results}\label{sec:Proof-Main-results-abstract}

This section is devoted to prove the main results of this paper. 
The first  result provides a characterization of the spectral stability of a relative equilibrium $\vz$  in terms of a spectral condition on the central configuration $\vxi$.\\

{\bf Proof of Theorem \ref{thm:generalization-3-1-rob-intro}.\/}
The proof of this result result direct follows by the first claim in   Lemma \ref{thm:spectrum-B-Hessian}. In fact, $\mu$ is an eigenvalue of $M_\Gamma^{-1}D^{2}H(\vxi)$ if and only if $\lambda= \sqrt{\mu^2-\omega^2}$ is an eigenvalue of $B$. By definition, $\vz$ is spectrally stable if and only if the spectrum of $B$ is purely imaginary or which is the same that $\mu^2-\omega^2 \leq 0$. This concludes the proof.\qed 

The next result shed provides a cler relation intertwining the spectral condition on the central configuration generating the relative equilibrium seen as critical point of the Hamiltonian on the shape pseudo-sphere and the dynamical (stability) properties of it. Roberts  in Theorem 3.3 of  \cite{Rob13}  characterizes linearly stable relative equilibria in terms of the minimality properties that the central configuration (originating such an equilibrium) possesses. 
\begin{thm}[\bf{Roberts 2013}]\label{thm:roberts2013}
We assume that for every $j$, $\Gamma_j>0$. Then a relative equilibrium $\vz$ is linearly stable if and only if it is a non-degenerate minimum of $H$ restricted to the shape-pseudo-sphere.
\end{thm}
Thus, by Theorem \ref{thm:roberts2013},  the linear stability of a relative equilibrium is equivalent to the fact that the central configuration generating it has a vanishing Morse index and it is non-degenerate (meaning that the kernel dimension of the Hessian of the Hamiltonian restricted to the shape pseudo-sphere vanishes identically). 
However this result is valid only under the assumption that all  circulations have the same sign. In Theorem \ref{thm:stability-Morse-index-intro},  by using the analysis performed in the previous sections, we are able to remove the condition on the circulations' sign admitting any kind of (non-vanishing) circulation and we provide a relation between the spectral stability of a  relative equilibrium and the Morse index  of the central configuration  generating it. As Corollary of this result, we complement the aforementioned Theorem \ref{thm:roberts2013}. 

Before giving the proof of this result, we observe that if all circulations strengths have all  the same sign (for instance, positive), then $M_\Gamma$ is positive definite (thus $\iMor(M_\Gamma)=0$) and by Equation \eqref{eq:segno-omega} in particular $\omega$ is positive. Thus by the first claim of Theorem \ref{thm:stability-Morse-index-intro}, we conclude that $\vxi$ is a minimum (maybe degenerate).  
\begin{cor}	If $\Gamma_j>0$ for all $j$, and we assume that  $\vz$ is a spectral  stable non-degenerate relative equilibrium. Then the central configuration $\vxi$   is a (maybe degenerate) minimum of $H$.
	\end{cor}
\begin{proof}
The proof of the first claim follows by the above discussion. 	
\end{proof}
Before  providing the proof of Theorem \ref{thm:stability-Morse-index-intro}, we start  proving  the following technical result.   \begin{lem}\label{thm:Morseindex-of -A-on-general-eigenvector-subspace-for-real-eigenvalue}
			Let $\nu$ be a non-zero real eigenvalue of matrix $A_\Gamma$, then we have that \[
			\begin{cases}
			n_{-}\left(\left.\hat{A}_{\Gamma}\right|_{I_{\nu}}\right)=n_{-}\left(\left.M_{\Gamma}\right|_{I_{\nu}}\right) & \textrm{if}\ \nu>0\\
			n_{-}\left(\left.\hat{A}_{\Gamma}\right|_{I_{\nu}}\right)=n_{+}\left(\left.M_{\Gamma}\right|_{I_{\nu}}\right) & \textrm{if}\ \nu<0
			\end{cases}
			\]
	\end{lem}
\begin{proof}
	 We assume that there exist   two different Jordan blocks $\Gamma_{1}(\nu),\Gamma_{2}(\nu)$ corresponding to same eigenvalue $\nu$ and, as before, we  denote by  $\{\vv_{i}\}_{i=0}^{p}$ and $\{\vw_{j}\}_{j=0}^{q}$ the 
	Jordan chains corresponding to these   Jordan blocks.  Let us consider the following matrix block decomposition
	\[
	\widehat{A}(\nu)=\begin{bmatrix}A_{1} & A_{2}\\
	A_{2}^{T} & A_{3}
	\end{bmatrix}\quad 
	\textrm{ and }\quad 
	M_\Gamma(\nu)=\begin{bmatrix}M_{1} & M_{2}\\
	M_{2}^{T} & M_{3}
	\end{bmatrix},
	\]
	where $(A_{1})_{ij}=\left\langle \widehat{A}	\vv_{i},\vv_{j}\right\rangle $,
	$(A_{2})_{ij}=\left\langle \widehat{A}\vv_{i},\vw_{j}\right\rangle ,$ $(A_{3})_{ij}=\left\langle \widehat{A}\vw_{i},\vw_{j}\right\rangle$, $\ (M_{1})_{ij}=\left\langle M\vv_{i},\vv_{j}\right\rangle$, $(M_{2})_{ij}=\left\langle M\vv_{i},\vw_{j}\right\rangle $
	and finally $(M_{3})_{ij}=\left\langle M\vw_{i},\vw_{j}\right\rangle .$ By Lemma \ref{thm:matrix-1}, one immediately get that the $p \times p$ block $A_1$ is given by:  
	\small{\begin{alignat}{1}\label{eq:representionofA-1}
	&A_{1}=  \left\langle \widehat{A}\vv_{i+1},\vv_{j+1}\right\rangle 
	= \nu\left\langle M_\Gamma\vv_{i+1},\vv_{j+1}\right\rangle +\left\langle M_\Gamma\vv_{i},\vv_{j+1}\right\rangle  =\\
	& \begin{bmatrix}0 & 0 & \cdots & 0 & \nu\left\langle M_\Gamma\vv_{1},\vv_{p}\right\rangle \\
	0 & 0 & \cdots & \nu\left\langle M_\Gamma\vv_{1},\vv_{p}\right\rangle  & \left\langle M_\Gamma\vv_{1}+\nu M_\Gamma\vv_{2},\vv_{p}\right\rangle \\
	\vdots & \vdots & \ddots & \vdots & \vdots\\
	0 & \nu\left\langle M_\Gamma\vv_{1},\vv_{p}\right\rangle  & \cdots & \left\langle M_\Gamma\vv_{p-3}+\nu M_\Gamma\vv_{p-2},\vv_{p}\right\rangle  & \left\langle M_\Gamma\vv_{p-2}+\nu M_\Gamma\vv_{p-1},\vv_{p}\right\rangle \\
	\nu\left\langle M_\Gamma\vv_{1},\vv_{p}\right\rangle  & \left\langle M_\Gamma\vv_{1}+\nu M_\Gamma\vv_{2},\vv_{p}\right\rangle  & \cdots & \left\langle M_\Gamma\vv_{p-2}+\nu M_\Gamma\vv_{p-1},\vv_{p}\right\rangle  & \left\langle M_\Gamma\vv_{p-1}+\nu M_\Gamma\vv_{p},\vv_{p}\right\rangle 
	\end{bmatrix}
	\end{alignat}\/}
	and 
	\begin{equation}\label{eq:representionofM-1}
	M_{1}=\begin{bmatrix}0 & 0 & \cdots & 0 & \left\langle M_\Gamma \vv_{1},\vv_{p}\right\rangle \\
	0 & 0 & \cdots & \left\langle M_\Gamma \vv_{1},\vv_{p}\right\rangle  & \left\langle M_\Gamma v_{2},\vv_{p}\right\rangle \\
	\vdots & \vdots & \ddots & \vdots & \vdots\\
	0 & \left\langle M_\Gamma \vv_{1},\vv_{p}\right\rangle  & \cdots & \left\langle M_\Gamma \vv_{p-2},\vv_{p}\right\rangle  & \left\langle M_\Gamma\vv_{p-1},\vv_{p}\right\rangle \\
	\left\langle M_\Gamma \vv_{1},\vv_{p}\right\rangle  & \left\langle M_\Gamma \vv_{2},\vv_{p}\right\rangle  & \cdots & \left\langle M_\Gamma \vv_{p-1},\vv_{p}\right\rangle  & \left\langle M_\Gamma \vv_{p},\vv_{p}\right\rangle 
	\end{bmatrix}.
	\end{equation}
	It is readily seen that the matrix $A_1$ given in Equation \eqref{eq:representionofA-1} can be written in equivalent form, as follows 
	\begin{equation}\label{eq:A1}
	A_{1}= M_{1}\Gamma_{p}(\nu,1).
	\end{equation}
	Analogously, we have that 
	\begin{equation}\label{eq:A3}
	A_{3}=M_{3}\Gamma_{q}(\nu,1).
	\end{equation}
	By  Lemma \ref{thm:matrix-2}, one gets that the $p\times q$ block $A_2$ is given by 
	\small{\begin{alignat}{1}\label{eq:matrix3}
	&A_{2}  = \left\langle \widehat{A}\vv_{i+1},\vw_{j+1}\right\rangle 
	=  \nu\left\langle M_\Gamma\vv_{i+1},\vw_{j+1}\right\rangle +\left\langle M_\Gamma\vv_{i},\vw_{j+1}\right\rangle= \\
	&  \begin{bmatrix}0 & 0 & \cdots & 0 & \nu\left\langle M_\Gamma\vv_{1},\vv_{q}\right\rangle \\
	0 & 0 & \cdots & \nu\left\langle M_\Gamma\vv_{1},\vv_{q}\right\rangle  & \left\langle M_\Gamma\vv_{1}+\nu M_\Gamma\vv_{2},\vv_{q}\right\rangle \\
	\vdots & \vdots & \ddots & \vdots & \vdots\\
	0 & \nu\left\langle M_\Gamma\vv_{1},\vv_{q}\right\rangle  & \cdots & \left\langle M_\Gamma\vv_{p-3}+\nu M_\Gamma\vv_{p-2},\vv_{q}\right\rangle  & \left\langle M_\Gamma\vv_{p-2}+\nu M_\Gamma\vv_{p-1},\vv_{q}\right\rangle \\
	\nu\left\langle M_\Gamma\vv_{1},\vv_{q}\right\rangle  & \left\langle M_\Gamma\vv_{1}+\nu M_\Gamma\vv_{2},\vv_{q}\right\rangle  & \cdots & \left\langle M_\Gamma\vv_{p-2}+\nu M_\Gamma\vv_{p-1},\vv_{q}\right\rangle  & \left\langle M_\Gamma\vv_{p-1}+\nu M_\Gamma\vv_{p},\vv_{q}\right\rangle 
	\end{bmatrix}
	\end{alignat}}
	and 
	\begin{equation}\label{eq:matrix4}
	M_{2}=\begin{bmatrix}0 & 0 & \cdots & 0 & \left\langle M_\Gamma\vv_{1},\vv_{q}\right\rangle \\
	0 & 0 & \cdots & \left\langle M_\Gamma\vv_{1},\vv_{q}\right\rangle  & \left\langle M_\Gamma\vv_{2},\vv_{q}\right\rangle \\
	\vdots & \vdots & \ddots & \vdots & \vdots\\
	0 & \left\langle M_\Gamma\vv_{1},\vv_{q}\right\rangle  & \cdots & \left\langle M_\Gamma\vv_{p-2},\vv_{q}\right\rangle  & \left\langle M_\Gamma\vv_{p-1},\vv_{q}\right\rangle \\
	\left\langle M_\Gamma\vv_{1},\vv_{q}\right\rangle  & \left\langle M_\Gamma\vv_{2},\vv_{q}\right\rangle  & \cdots & \left\langle M_\Gamma\vv_{p-1},\vv_{q}\right\rangle  & \left\langle M_\Gamma\vv_{p},\vv_{q}\right\rangle 
	\end{bmatrix}
	\end{equation}
	so, as before, by Equations \eqref{eq:matrix3} and \eqref{eq:matrix4} imply that 
	\begin{equation}\label{eq:A2}
	A_{2}=M_{2}\varGamma_{q}(\nu,1).
	\end{equation}
	Similarly for the term 
	\begin{equation}
	\trasp{A}_{2}=M_{2}^{T}\varGamma_{p}(\nu,1).\label{eq:A2T}
	\end{equation}
	Thus \eqref{eq:A1}, \eqref{eq:A3}, \eqref{eq:A2} and finally \eqref{eq:A2T} imply that 
	\begin{equation}\label{eq:representationofv}
	\widehat A_\Gamma^\nu =M_\Gamma^\nu\,\textrm{diag}\big[\Gamma_{p}(\nu,1),\ \Gamma_{q}(\nu,1)\big].
	\end{equation}
\paragraph*{Case 1.}If $\nu>0,$ we define the (analytic) path of symmetric matrices pointwise given by  
\[
f(t):=M_{\Gamma}^{\nu}\begin{pmatrix}\Gamma_{p}(t,1) & 0\\
0 & \Gamma_{q}(t,1)
\end{pmatrix}
\]
parametrized by the interval $ [1,\nu]$, if $\nu  >1$ and by $[\nu,1]$, if $\nu < 1$.  Moreover, we let 

\[
g(s):=M_{\Gamma}^{\nu}\begin{pmatrix}\Gamma_{p}(1,s) & 0\\
0 & \Gamma_{q}(1,s)
\end{pmatrix}  \quad  \textrm{ for } s \in [0,1].
\]
 
If an eigenvalue of $f(t)$ (resp. $g(s)$) changes sign, than $\det f(t)=0$ (resp. $\det(g(s)=0$). However, it is immediate to see that this cannot occur.   We observe that the composition of the two  paths $f$ and $g$ is a continuous path joining the matrices  $M_\Gamma^\nu$ matrix to $\widehat A_\Gamma^\nu$. By this argument it  then follows that both matrices belong to the same connected component and in particular the inertia indices coincide; thus in symbols, we have
\begin{align*}
\iMor\big(\widehat{A}(\nu)\big) & =\iMor\big(M_\Gamma(\nu)\big),\\
\coiMor\big(\widehat{A}(\nu)\big) & =\coiMor\big(M_\Gamma(\nu)\big).
\end{align*}

	\paragraph{Case 2.} If $\nu<0$, as before,
	we define the path of symmetric matrices 
\[
f(t):=M_{\Gamma}^{\nu}\begin{pmatrix}\Gamma_{p}(t,1) & 0\\
0 & \Gamma_{q}(t,1)
\end{pmatrix}
\]
	parametrized by the interval $[\nu,-1]$ if $\nu < -1$ and by $[-1,\nu]$ if $\nu > -1$. As before, we let
	\[
	g(s):=M_{\Gamma}^{\nu}\begin{pmatrix}\Gamma_{p}(-1,-s) & 0\\
	0 & \Gamma_{q}(-1,-s)
	\end{pmatrix}
	\]
	where $s\in[-1,0]$. Arguing as before, we get  
	\begin{align*}
	\iMor\big(\widehat A_\Gamma ^\nu\big) & =\coiMor\big(M_\Gamma^\nu\big),\\
	\coiMor\big(\widehat A_\Gamma^\nu\big) & =\iMor\big(M_\Gamma^\nu\big).
	\end{align*}
\end{proof}

{\bf Proof of  Theorem \ref{thm:stability-Morse-index-intro}.\/}
Since $\vz$ is non-degenerate and spectrally stable relative equilibrium,  then by invoking  Lemma \ref{thm:spectrum-B-Hessian}, we get that 
\[
\sigma(A_\Gamma)\subset \R\bigcup\left\{\omega+ix|x\in \R\right\}.
\]
We notice that 
\[
\widehat A_\Gamma(\vxi)=
D^{2}H(\vxi)\vxi+\omega M_{\Gamma}\vxi=M_{\Gamma}\left(M_{\Gamma}^{-1}D^{2}H(\vxi)+\omega \Id\right)\vxi=0
\]
and from property (iii) of Lemma \ref{thm:property-of-H}  we have
\begin{align*}
\widehat A_\Gamma(K\vxi)=D^{2}H(\vz)K\vxi+\omega M_{\Gamma}K\vxi & =M_{\Gamma}\left(M_{\Gamma}^{-1}D^{2}H(\vz)+\omega \Id\right)K\vxi\\
& =M_{\Gamma}K\left(-M_{\Gamma}^{-1}D^{2}H(\vz)+\omega \Id\right)\xi\\
& =M_{\Gamma}K\left(-M_{\Gamma}^{-1}D^{2}H(\vz)-\omega I+2\omega I\right)\vxi\\
& =2\omega M_{\Gamma}K\vxi.
\end{align*}
If  $W=\textup{span}(\vxi, K\vxi)$, then we have  
\begin{align*}
\left.\widehat{A}_\Gamma\right|_{W} & =\begin{bmatrix}\langle \widehat A_\Gamma(\vxi), \vxi\rangle & \langle \widehat A_\Gamma(\vxi), K \vxi \rangle\\
\langle \widehat A_\Gamma(K\vxi), \vxi\rangle & \langle \widehat A_\Gamma(K\vxi),K \vxi\rangle
\end{bmatrix}\\
& =\begin{bmatrix}0 & 0\\
\langle K\vxi,\widehat A_\Gamma(\vxi)\vxi\rangle & \langle-KD^{2}H(\vxi)\vxi+\omega M_{\Gamma}K\vxi,K\vxi\rangle
\end{bmatrix}\\
& =\begin{bmatrix}0 & 0\\
0 & 2\omega\langle M_{\Gamma}\xi,\xi\rangle
\end{bmatrix}.
\end{align*}
Thus 
\begin{equation}
\iMor(\widehat{A}(\vxi)|_{W})=\begin{cases}
1 & \textrm{ if }\ 2\omega\left\langle M_{\Gamma}\xi,\xi\right\rangle <0,\\
0 & \textrm{ if }\ 2\omega\left\langle M_{\Gamma}\xi,\xi\right\rangle >0.
\end{cases}
\end{equation}
In conclusion, we get 
\[
\iMor(\widehat{A}(\vxi)|_{W})=\textup{sign}(2\omega\left\langle M_\Gamma \vxi,\vxi\right\rangle )
\]
where $\textup{sign}$ stands for the sign.  Let us now consider the $\hat{A}_\Gamma$-orthogonal $M_\Gamma$-orthogonal invariant decomposition of the full (complexified) phase space $\C^{2n}$.   Since
$\C^{2n}=W\oplus W^{\perp}=I_{0}\oplus I_{2\omega}\oplus_{i}(I_{\nu_{i}}\oplus_{i=1}^l I_{\nu_{2\omega-\nu_{i}}})$,
where $\nu_{i}\in\sigma(A_\Gamma)\setminus\{0,2\omega\}.$
So we have that 
\begin{equation}\label{eq:inex 1}
\iMor(\widehat{A}(\vxi)|_{W^{\perp}})=\sum_{i=1}^l\left(\iMor(\widehat{A}(\vxi)|_{I_{\nu_{i}}})+\iMor(\widehat{A}(\vxi)|_{I_{2\omega-\nu_{i}}})\right)
\end{equation}
and
\begin{equation}\label{eq:index 2}
\iMor(M_\Gamma|_{W^{\perp}})=\sum_{i=1}^l\left(\iMor(M_\Gamma|_{I_{\nu_{i}}})+\iMor(M_\Gamma|_{I_{2\omega-\nu_{i}}})\right).
\end{equation}
We let  $\nu\in\sigma(A_\Gamma)\backslash\{0,2\omega\}$ and we consider separately the two cases: $\nu \in \R$ and  $\nu \in\C\setminus \R$. 

If $\nu \in \R$ as consequence of the first item in  Lemma \ref{thm:spectrum-B-Hessian} only two cases can occur: 
\begin{itemize}
	\item Subcase 1: $\omega$ positive and $0<\nu<2\omega$. (In particular $2\omega-\nu >0$ and $\nu>0$);
	\item Subcase 2: $\nu \in \R$, $\omega$ negative and $2\omega<\nu<0$. (In particular $2\omega -\nu <0$ and $\nu<0$).
\end{itemize} 
By Lemma \ref{thm:Morseindex-of -A-on-general-eigenvector-subspace-for-real-eigenvalue}, we have that\[
		\begin{cases}
		n_{-}\left(\left.\hat{A}_{\Gamma}\right|_{I_{\nu}}\right)=n_{-}\left(\left.M_{\Gamma}\right|_{I_{\nu}}\right) & \quad \textrm{ [Subcase 1] }\\
		n_{-}\left(\left.\hat{A}_{\Gamma}\right|_{I_{\nu}}\right)=n_{+}\left(\left.M_{\Gamma}\right|_{I_{\nu}}\right) & \quad \textrm{ [Subcase 2] }
		\end{cases}
		\]

\paragraph{Case 2: $\nu\in\left\{ ix+\omega|x\in \R\right\} \subset \C\setminus \R$.}
By using  Lemma \ref{thm:positive =negative}, we already know that 
$\iMor(\widehat A_\Gamma^\nu)=m_{\nu},$ where $m_{\nu}$ is the algebraic
multiplicities of the eigenvalue $\nu$. Moreover, from Corollary \ref{eq:orthogonal}, the matrix  $M_\Gamma^\nu$
can be represented as follows 
\[
\begin{bmatrix}0 & Y\\
\trasp{Y} & 0
\end{bmatrix},
\]
for some $Y \in \Lin(E_\nu, E_{\bar \nu})$.

Now, since $M_\Gamma^\nu$ is non-degenerate,  it readily follows that  $\iMor\big(M_\Gamma^\nu\big)=\coiMor\big(M_\Gamma^\nu\big)=m_{\nu}$. In conclusion, we get
 $\iMor(\widehat A^\nu)=\iMor\big(M_\Gamma^\nu\big)=\coiMor\big(M^\nu\big)=m_{\nu}.$
This concludes the proof. \qed

%%%%%%%%%%%%%%%%%%%%%%%%%%%%%%%%%%%%%%%%%%%%%%%%%%%%%%%%%%%%
%%%%%%%
%%%%%%%
%%%%%%%
%%%%%%%%%%%%%%%%%%%%%%%%%%%%%%%%%%%%%%%%%%%%%%%%%%%%%%%%%%%%

\section{Some symmetric examples}

This section is devoted to the application of Theorem \ref{thm:stability-Morse-index-intro} to some specific examples of relative equilibria in the planar $N$-vortex problem. 

\subsection{The equilateral triangle}  We begin with the well-known equilateral triangle solution in the three-vortex problem. Placing three vortices of any strength at the vertices of an equilateral triangle yields a relative equilibrium. 	Synge \cite{Syn49}, showed that the corresponding relative equilibrium is  linearly stable if and only if $L>0$. (Cfr. \cite[Section 4]{Rob13} for further details).
\begin{thm}
Given any three circulations $\Gamma_{1},\ \Gamma_{2}\ \Gamma_{3}$,  let 
\[
\widehat{\vxi}_{1}=(1,0),\ \widehat{\vxi}_{2}=(-\dfrac{1}{2},\dfrac{\sqrt{3}}{2}),\ \widehat{\vxi}_{3}=(-\dfrac{1}{2},-\dfrac{\sqrt{3}}{2}),
\]
and let $\widehat{\vc}=\sum_{i=3}^{3}\Gamma_{i}\widehat{\vz}_{i}$. We assume that  $L>0$ and  let $\vxi=(\vxi_{1},\vxi_{2},\vxi_{3})$, for  $\vxi_{i}=\widehat{\vxi}_{i}-\widehat{\vc}$. Then 
\[
\iMor(\widehat A_\Gamma(\vxi))=
\begin{cases} 0  & \textrm{ if } \Gamma_{1}, \Gamma_{2}, \Gamma_{3} \textrm { have the same sign }\\
1  & \textrm{ if there is only one } \Gamma_{i}<0\\
2  & \textrm{ otherwise }
\end{cases}.
\]
\end{thm}
\begin{proof}
As proved by author in \cite{Syn49},  we get that  $\vxi=(\vxi_{1},\vxi_{2},\vxi_{3})$ is a linearly stable iff $L>0$; moreover, the angular velocity $\omega=\frac{\Gamma}{3}$. We note that the vortex positions $\vxi_{i}=\widehat{\vxi}_{i}-\widehat{\vc}$ have
center of vorticity at the origin.

By an explicit calculation, we get
\begin{multline}
\widehat{\vc}  =\dfrac{1}{\Gamma}\sum_{i=1}^{3}\Gamma_{i}\vxi_{i}
  =\dfrac{1}{\Gamma}\left(\Gamma_{1}\left(1,0\right)+\Gamma_{2}\left(-\frac{1}{2},\frac{\sqrt{3}}{2}\right)+\Gamma_{3}\left(-\frac{1}{2},-\frac{\sqrt{3}}{2}\right)\right)\\
  =\left(\frac{\Gamma_{1}}{\Gamma}-\frac{\Gamma_{1}+\Gamma_{2}}{2\Gamma},\frac{\sqrt{3}}{2\Gamma}\left(\Gamma_{2}-\Gamma_{3}\right)\right)
\end{multline}

\begin{multline}
\vxi_{1}  =\widehat{\vxi}_{1}-\widehat{\vc}
 =\left(1,0\right)-\left(\dfrac{\Gamma_{1}}{\Gamma}-\dfrac{\Gamma_{3}+\Gamma_{2}}{2\Gamma},\dfrac{\sqrt{3}}{2\Gamma}\left(\Gamma_{2}-\Gamma_{3}\right)\right)\\
 =\left(\dfrac{3\left(\Gamma_{3}+\Gamma_{2}\right)}{2\Gamma},-\dfrac{\sqrt{3}}{2\Gamma}\left(\Gamma_{2}-\Gamma_{3}\right)\right)
\end{multline}

\begin{multline}
\vxi_{2}  =\widehat{\vxi}_{2}-\widehat{\vc}
  =\left(-\dfrac{1}{2},\dfrac{\sqrt{3}}{2}\right)-\left(\dfrac{\Gamma_{1}}{\Gamma}-\dfrac{\Gamma_{1}+\Gamma_{2}}{2\Gamma},\dfrac{\sqrt{3}}{2\Gamma}\left(\Gamma_{2}-\Gamma_{3}\right)\right)\\
  =\left(-\dfrac{1}{2}-\dfrac{\Gamma_{1}}{\Gamma}+\dfrac{\Gamma_{1}+\Gamma_{2}}{2\Gamma},\dfrac{\sqrt{3}}{2}-\dfrac{\sqrt{3}}{2\Gamma}\left(\Gamma_{2}-\Gamma_{3}\right)\right)\\
  =\left(-\dfrac{3\Gamma_{1}}{2\Gamma},\dfrac{\sqrt{3}}{2\Gamma}\left(\Gamma_{1}+2\Gamma_{3}\right)\right)
\end{multline}

\begin{multline}
\vxi_{3}  =\widehat{\vxi}_{3}-\widehat{\vc}
  =\left(-\dfrac{1}{2},-\dfrac{\sqrt{3}}{2}\right)-\left(\dfrac{\Gamma_{1}}{\Gamma}-\dfrac{\Gamma_{1}+\Gamma_{2}}{2\Gamma},\dfrac{\sqrt{3}}{2\Gamma}\left(\Gamma_{2}-\Gamma_{3}\right)\right)\\
  =\left(-\dfrac{1}{2}-\frac{\Gamma_{1}}{\Gamma}+\dfrac{\Gamma_{1}+\Gamma_{2}}{2\Gamma},-\dfrac{\sqrt{3}}{2}-\dfrac{\sqrt{3}}{2\Gamma}\left(\Gamma_{2}-\Gamma_{3}\right)\right)\\
  =\left(-\dfrac{3\Gamma_{1}}{2\Gamma},-\dfrac{\sqrt{3}}{2\Gamma}\left(\Gamma_{1}+2\Gamma_{2}\right)\right).
\end{multline}
Summing up all computations we get 
\begin{multline}
\left\langle M_\Gamma \vxi,\vxi\right\rangle   =
\begin{bmatrix}
\dfrac{\Gamma_{1}}{\Gamma}-\dfrac{\Gamma_{1}+\Gamma_{2}}{2\Gamma}\\
\dfrac{\sqrt{3}}{2\Gamma}\left(\Gamma_{2}-\Gamma_{3}\right)\\
-\dfrac{3\Gamma_{1}}{2\Gamma}\\
\dfrac{\sqrt{3}}{2\Gamma}\left(\Gamma_{1}+2\Gamma_{3}\right)\\
-\dfrac{3\Gamma_{1}}{2\Gamma}\\
-\dfrac{\sqrt{3}}{2\Gamma}\left(\Gamma_{1}+2\Gamma_{2}\right)
\end{bmatrix}^{T}\cdot
\begin{bmatrix}\Gamma_{1} & 0 & 0 & 0 & 0 & 0\\
0 & \Gamma_{1} & 0 & 0 & 0 & 0\\
0 & 0 & \Gamma_{2} & 0 & 0 & 0\\
0 & 0 & 0 & \Gamma_{2} & 0 & 0\\
0 & 0 & 0 & 0 & \Gamma_{3} & 0\\
0 & 0 & 0 & 0 & 0 & \Gamma_{3}
\end{bmatrix}\cdot
\begin{bmatrix}
\dfrac{\Gamma_{1}}{\Gamma}-\dfrac{\Gamma_{1}+\Gamma_{2}}{2\Gamma}\\
\dfrac{\sqrt{3}}{2\Gamma}\left(\Gamma_{2}-\Gamma_{3}\right)\\
-\dfrac{3\Gamma_{1}}{2\Gamma}\\
\dfrac{\sqrt{3}}{2\Gamma}\left(\Gamma_{1}+2\Gamma_{3}\right)\\
-\dfrac{3\Gamma_{1}}{2\Gamma}\\
-\dfrac{\sqrt{3}}{2\Gamma}\left(\Gamma_{1}+2\Gamma_{2}\right)
\end{bmatrix}\\
 ={\dfrac{\left(3\,\Gamma_{2}+3\,\Gamma_{3}\right)^{2}\Gamma_{1}}{4{\Gamma}^{2}}}
 +
 {\dfrac{3\left(\Gamma_{2}-\Gamma_{3}\right)^{2}\Gamma_{1}}{4{\Gamma}^{2}}}
 +{\dfrac{9{\Gamma_{1}}^{2}\Gamma_{2}}{4{\Gamma}^{2}}}
 +{\dfrac{3\left(\Gamma_{1}+2\,\Gamma_{3}\right)^{2}\Gamma_{2}}{4{\Gamma}^{2}}}\\
 +{\dfrac{9{\Gamma_{1}}^{2}\Gamma_{3}}{4{\Gamma}^{2}}}
 +{\dfrac{3\left(\Gamma_{1}+2\,\Gamma_{2}\right)^{2}\Gamma_{3}}{4{\Gamma}^{2}}}\\
  =\dfrac{3}{4\Gamma^{2}}\left(3\left(\,\Gamma_{2}+\,\Gamma_{3}\right)^{2}\Gamma_{1}+\left(\Gamma_{2}-\Gamma_{3}\right)^{2}\Gamma_{1}+3{\Gamma_{1}}^{2}\Gamma_{2}+\left(\Gamma_{1}+2\,\Gamma_{3}\right)^{2}\Gamma_{2}\right.\\\left.+3{\Gamma_{1}}^{2}\Gamma_{3}+\left(\Gamma_{1}+2\,\Gamma_{2}\right)^{2}\Gamma_{3}\right)\\
  =\dfrac{3}{\Gamma^{2}}\left(\Gamma\left(\Gamma_{1}^{2}+\Gamma_{2}^{2}+\Gamma_{3}^{2}\right)+3\Gamma_{1}\Gamma_{2}\Gamma_{3}\right).
\end{multline}
Let $a\=\Gamma\left(\Gamma_{1}^{2}+\Gamma_{2}^{2}+\Gamma_{3}^{2}\right)+3\Gamma_{1}\Gamma_{2}\Gamma_{3}$,
then the signature  of the quadratic form $\left\langle M_\Gamma \vxi,\vxi\right\rangle $ coincides with that of the quadratic form  $a$.

We distinguish the following four cases.\\
\paragraph{First case.} We assume that $\Gamma_{1},\ \Gamma_{2},\ \Gamma_{3}>0$. In this case we have  $\omega>0$ and $\left\langle M_\Gamma\vxi,\vxi\right\rangle >0$ and so, according to Theorem \ref{thm:stability-Morse-index-intro}, we get that $\iMor(\widehat{A}_\Gamma)=0.$

\paragraph{Second case.} We assume that  $\Gamma_{1},\ \Gamma_{2},\ \Gamma_{3}<0$. Thus 
we have $\omega<0$ and $\left\langle M_\Gamma \vxi,\vxi\right\rangle <0$; so,  according to Theorem \ref{thm:stability-Morse-index-intro}, we get  that $\iMor(\widehat{A}_\Gamma)=0$.

\paragraph{Third case.} We assume that $\Gamma_{i}<0$ only for one index $i$. In this case,  without losing in 
 generalities, we may assume that $\Gamma_{3}<0$. Since $L>0$, then we have that 
 \[
 -\dfrac{\Gamma_{1}\Gamma_{2}}{\Gamma_{1}+\Gamma_{2}}<\Gamma_{3}<0.
 \]
We claim  that $\left\langle M_\Gamma \vxi,\vxi\right\rangle >0$. In order to  prove this , as already observed, it is enough to prove that $a>0$.  We observe that $a$ could be re-written as follows
\begin{align*}
a & =\Gamma\left(\Gamma_{1}^{2}+\Gamma_{2}^{2}+\Gamma_{3}^{2}\right)+3\Gamma_{1}\Gamma_{2}\Gamma_{3}\\
& =\left(\Gamma_{1}+\Gamma_{2}+\Gamma_{3}\right)\left(\Gamma_{1}^{2}+\Gamma_{2}^{2}+\Gamma_{3}^{2}\right)+3\Gamma_{1}\Gamma_{2}\Gamma_{3}\\
& =\Gamma_{1}^{3}\left(\left(1+\dfrac{\Gamma_{2}}{\Gamma_{1}}+\frac{\Gamma_{3}}{\Gamma_{1}}\right)\left(1+\dfrac{\Gamma_{2}^{2}}{\Gamma_{1}^{2}}+\dfrac{\Gamma_{3}^{2}}{\Gamma_{1}^{2}}\right)+3\dfrac{\Gamma_{2}\Gamma_{3}}{\Gamma_{1}^{2}}\right).
\end{align*}

We let 
\[
x=\dfrac{\Gamma_{2}}{\Gamma_{1}},\ y=\dfrac{\Gamma_{3}}{\Gamma_{1}}.
\]
Since $\Gamma_{1}>0$, Thus the signature of the quadratic form $a$ agrees with the sign of  the function $b$ defined below
\[
b(x,y)\=(1+x+y)(1+x^2+y^2)+3xy,
\]
where  $x>0,\ -\dfrac{x}{1+x}<y<0.$

Fix $x\in[0,+\infty),$ differentiating $(1+x+y)(1+x^{2}+y^{2})+3xy$
with respect to $y$, yields 
\[
3y^{2}+2y(x+1)+3x+x^{2}+1=3\left(y+\dfrac{x+1}{3}\right)^{2}+\dfrac{2x^{2}+7x+2}{3}>0.
\]
This implies that $y\mapsto b(x,y)$ is a monotone increasing function (with respect to $y$) thus the infimum is getting precisely at    $-\dfrac{x}{1+x}$. Setting $y=-\dfrac{x}{1+x}$, then we get 
\begin{multline}
\left(1+x-\dfrac{x}{1+x}\right)\left(1+x^{2}+\dfrac{x^{2}}{(1+x)^{2}}\right)-3\dfrac{x^{2}}{1+x}\\=\dfrac{1}{(1+x)^{3}}\left(x^{6}+3x^{5}+3x^{4}+x^{3}+3x^{2}+3x+1\right)>0\quad \forall\, x \in [0,+\infty).
\end{multline}
By this, we immediately get that 
%
%$\left(1+x-\frac{x}{1+x}\right)\left(1+x^{2}+\frac{x^{2}}{(1+x)^{2}}\right)-3\frac{x^{2}}{1+x}>0$
%for each $x\in[0,+\infty),$ this implies 
%\[
%\left(1+x-\dfrac{x}{1+x}\right)\left(1+x^{2}+\dfrac{x^{2}}{(1+x)^{2}}\right)-3\dfrac{x^{2}}{1+x}>0 \quad \forall\, x \in [0,+\infty).
%\]
%
for every $x\in(0,+\infty)$ and $-\dfrac{x}{1+x}<y<0$,  the function $b$  is positive or which is the same that the quadratic form $a$ is positive definite, hence  $\left\langle M_\Gamma \vxi,\vxi\right\rangle >0$.

As direct consequence of Equation \eqref{eq:segno-omega} as well as of the fact that $L >0$ and $\left\langle M_\Gamma \vxi,\vxi\right\rangle >0$, we get that $\omega >0$. So now, according to Theorem \ref{thm:stability-Morse-index-intro}, we have that $\iMor(\widehat{A}_\Gamma)=\iMor(M_\Gamma)-1=1.$

\paragraph{Fourth case.} We assume now that  $\Gamma_{i}>0$  only for one index $i$. Arguing precisely as before we get $\left\langle M_\Gamma\vxi,\vxi\right\rangle <0$. In this case, however as direct consequence of  Equation \eqref{eq:segno-omega} s well as of the fact that $L >0$ and $\left\langle M_\Gamma \vxi,\vxi\right\rangle <0$, we get that $\omega<0$. 
According to Theorem \ref{thm:stability-Morse-index-intro}, we have that $\iMor(\widehat{A}_\Gamma)=\coiMor(M_\Gamma)=2.$
\end{proof}

\subsection{The rhombus families}
From paper \cite{Rob13}, we know that there exist two families of
relative equilibria where the configuration is a rhombus. Set $\Gamma_{1}=\Gamma_{2}=1$
and $\Gamma_{3}=\Gamma_{4}=m,$ where $m\in(-1,1]$ is a parameter.
Place the vortices at $z_{1}=(1,0)$, $z_{2}=(-1,0)$, $z_{3}=(0,y)$
and $z_{4}=(0,-y)$, forming a rhombus with diagonals lying on the
coordinate axis. This configuration is a central configuration provided that 
\begin{equation}\label{eq:rhombus-the-relationship-of-y-and-m}
y^{2}=\dfrac{1}{2}\left(\beta\pm\sqrt{\beta^{2}+4m}\right),\ \ \beta=3(1-m),
\end{equation}
or, equivalently,

\begin{equation}
m=\dfrac{3y^{2}-y^{4}}{3y^{2}-1}.
\end{equation}
The angular velocity is given by 
\[
\omega=\dfrac{m^{2}+4m+1}{2(1+my^{2})}=\dfrac{1}{2}+\dfrac{2m}{y^{2}+1}.
\]
The case $m<-1$ or $m>1$ can be reduced to this setup through a rescaling
of the circulations and a relabeling of the vortices. There are two
solutions depending on the sign choice for $y^{2}$. 

Taking plus sign in Equation \eqref{eq:rhombus-the-relationship-of-y-and-m} yields a solution for $m\in(-1,1]$ that always has $\omega>0$. We will call this solution
rhombus $A$. Taking $-$ in Equation \eqref{eq:rhombus-the-relationship-of-y-and-m} yields a solution for $m\in(-1,0)$ having  $\omega>0$ for $m\in(-2+\sqrt{3},0),$
but $\omega<0$ for $m\in(-1, -2+\sqrt{3})$. We will call this solution rhombus $B$. 

In the aforementioned paper, author computed the nontrivial eigenvectors of $M_\Gamma^{-1}D^{2}H(\vxi)$; in particular he proved that they are reals for every value of $m$ and are given by 
\begin{multline}\label{eq:v_1ev_2}
\vv_{1}=\trasp{[my,0,-my,0,0,-1,0,1]}\quad  \textrm{ and } K\vv_{1} \\
\vv_{2}=\trasp{[m,0,m,0,-1,0,-1,0]}\quad \textrm{ and } K\vv_{2}.
\end{multline}

\begin{thm}\label{1.Rhombus stable}{\bf (\cite[Theorem 4.1]{Rob13})\/} Under the previous notation, the following holds.
\begin{enumerate}
\item Rhombus $A$ is linearly stable for $-2+\sqrt{3}<m\le 1$ . At $m=-2+\sqrt{3}$ the relative equilibrium is degenerate. For $-1<m<-2+\sqrt{3}$ , rhombus $A$ is unstable and the nontrivial eigenvalues consist of a real pair and a pure imaginary pair. 
\item Rhombus $B$ is always unstable. One pair of eigenvalues is always real. The other pair of eigenvalues is purely imaginary for $-1<m<m^*$ and real for $m^*<m<0$ , where $m^*$ is the only real root of the cubic $9m^3+3m^2+7m+5$. At $m=m^*$, rhombus $B$ is degenerate.
\end{enumerate}
\end{thm}
As consequence of Theorem \ref{1.Rhombus stable} and   Theorem \ref{thm:stability-Morse-index-intro},  we get information on the Morse index of the rhombi configurations.
\begin{thm}
We assume that $\vz$ is the relative equilibrium generated by the rhombus  central configuration $\vxi$. Then 
\begin{enumerate}
\item if the central configuration corresponds to   rhombus $A$, then we have 
\[
\iMor(\widehat{A}_\Gamma(\vxi))=\begin{cases}
0 & \textrm{ if }\quad 0<m\le1,\\ 
3 & \textrm{ if }\quad -2+\sqrt{3}<m<0,\\
4 & \textrm{ if }\quad -1<m<-2+\sqrt{3}.
\end{cases}
\]
\item if the central configuration corresponds to  rhombus $B$, then we have 
\[
\iMor(\widehat{A}_\Gamma(\vxi))=
\begin{cases}
%0+0+0+1+1=
2 & \textrm{ if }\quad -2+\sqrt{3}<m<0,\\
%0+0+2+1+1=
4 & \text{ if }\quad m^{*}<m<-2+\sqrt{3},\\
%0+0+2+1+0=
3 & \textrm{ if }\quad -1<m<m^{*}.
\end{cases}
\]
where $m^*$ is the only real root of the cubic $9m^3+3m^2+7m+5$.
\end{enumerate}
\end{thm}
\begin{proof}
In order to prove the first claim, we observe that, by the computation performed by author in \cite[pag. 1129]{Rob13} as well as consequence of Lemma \ref{thm:invariant-subspaces} and Lemma \ref{thm:simple-eigenvalue block},  it follows that the spectrum  of the matrix $A_\Gamma$ is given by 
\begin{multline}\label{eq:mu1-mu2}
\sigma(A_\Gamma)=\Set{0, 2\omega, \omega,\omega-\mu_1, \omega+\mu_1, \omega-\mu_2, \omega+\mu_2} \quad \textrm{ where } 
\\ \mu_1=\dfrac{7y^4-18y^2+7}{2(y^2+1)(3y^2-1)} \textrm{ and } \mu_2=\dfrac{2(m+1)(1-y^2)}{(1+y^2)^2}=\dfrac{2(y^2-1)(y^2+2y-1)(y^2-2y-1)}{(y^2+1)^2(3y^2-1)}
\end{multline}
where the algebraic multiplicity of  $\omega$ is two.
Moreover 
\begin{multline}
	 A_{\Gamma}\vxi=0,\quad 
 A_{\Gamma}K\vxi=2\omega K\vxi,\quad  
 A_{\Gamma}\vs=\omega \vs,\quad 
 A_{\Gamma}K\vs=\omega \vs,\\ 
 A_{\Gamma}\vv_{1}=(\mu_{1}+\omega)\vv_{1},\quad 
 A_{\Gamma}K\vv_{1}=(-\mu_{1}+\omega)K\vv_{1},\\ 
 A_{\Gamma}\vv_{2}=(\mu_{2}+\omega)\vv_{2},\quad
 A_{\Gamma}K\vv_{2}=(-\mu_{2}+\omega)K\vv_{2}
\end{multline}
where $\vs$ is the vector defined in Lemma \ref{thm:invariant-subspaces}. Thus, by using Lemma \ref{thm:orthogonal-1}, we get the following  $M_\Gamma$-orthogonal direct sum decomposition  
\[
\C^{8}=I_{0}\oplus I_{2\omega}\oplus I_{\omega}\oplus I_{-\mu_{1}+\omega}\oplus I_{\mu_{1}+\omega}\oplus I_{\mu_{2}+\omega}\oplus I_{-\mu_{2}+\omega},
\]
 where 
 \begin{multline}
 I_{0}=\text{span}\{\vxi \},\ I_{2\omega}=\text{span}\{K\vxi\} ,\ I_{\omega}=\text{span}\{\vs,K\vs \},\ I_{\omega+\mu_1}=\text{span}\{\vv_1 \},\\ I_{\omega-\mu_1}=\text{span}\{K\vv_1 \},\ I_{\omega+\mu_2}=\text{span}\{\vv_2 \},\ I_{\omega-\mu_1}=\text{span}\{K\vv_2 \}.
 \end{multline}
 By a straightforward calculations, we get that 
\begin{multline}\label{eq:Rhombus-general-eigenvector-subspace}
\left.\widehat{A_\Gamma}\right|_{I_{0}}[\vxi]  =\left\langle M_{\Gamma}A_{\Gamma}\vxi,\vxi\right\rangle=0, \quad \left.\widehat{A_\Gamma}\right|_{I_{2\omega}}[K\vxi]  =\left\langle M_{\Gamma}A_{\Gamma}K\vxi,K\vxi\right\rangle =2\omega\left\langle M_{\Gamma}\vxi,\vxi\right\rangle,\\
\left.\widehat{A_\Gamma}\right|_{I_{\omega}}[\vs, K\vs]  =\begin{bmatrix}\left\langle M_{\Gamma}A_{\Gamma}\vs,\vs\right\rangle & \left\langle M_{\Gamma}A_{\Gamma}\vs,K\vs\right\rangle\\
\left\langle M_{\Gamma}A_{\Gamma}K\vs,\vs\right\rangle & \left\rangle M_{\Gamma}A_{\Gamma}K\vs,K\vs\right\rangle
\end{bmatrix}=\begin{bmatrix}\omega\left\langle M_{\Gamma}\vs,\vs\right) & \omega\left(M_{\Gamma}\vs,K\vs\right\rangle\\
\omega\left(M_{\Gamma}K\vs,\vs\right) & \omega\left\langle M_{\Gamma}\vs,\vs\right\rangle 
\end{bmatrix}\\  =\omega\begin{bmatrix}2+m & 0\\
0 & 2+m
\end{bmatrix},\quad 
\left.\widehat{A_\Gamma}\right|_{I_{\omega+\mu_{1}}}[\vv_1] =\left\langle M_{\Gamma}A_{\Gamma}\vv_{1},\vv_{1}\right\rangle=(\omega+\mu_{1})\left\langle M_{\Gamma}\vv_{1},\vv_{1}\right\rangle, \\
\left.\widehat{A_\Gamma}\right|_{I_{\omega-\mu_{1}}}[\vv_1] =\left\langle M_{\Gamma}A_{\Gamma}K\vv_{1},K\vv_{1}\right\rangle =(\omega-\mu_{1})\left\langle M_{\Gamma}\vv_{1},\vv_{1}\right\rangle,\\
\left.\widehat{A_\Gamma}\right|_{I_{\omega+\mu_{2}}}[\vv_2]  =\left\langle M_{\Gamma}A_{\Gamma}\vv_{2},\vv_{2}\right\rangle =(\omega+\mu_{2})\left\langle M_{\Gamma}\vv_{2},v\right\rangle,\\
\left.\widehat{A_\Gamma}\right|_{I_{\omega-\mu_{2}}}[\vv_2] =\langle M_{\Gamma}A_{\Gamma}K\vv_2,K\vv_{2}\rangle )=(\omega-\mu_{2})\left\langle M_{\Gamma}\vv_{2},\vv_{2}\right\rangle.
\end{multline}
By the rhombi classifications, we distinguish the two cases. 
\paragraph{Rhombus A central configuration.} We assume that the central configuration corresponds to rhombus $A$. Then, as already observed, we have
\[
y^{2}=\dfrac{1}{2}\left(3\left(1-m\right)+\sqrt{9\left(1-m\right)^{2}+4m}\right)
\]
and $\omega$ is always positive for all $m\in(-1,1]$. Since $M_\Gamma$ is the diagonal block matrix given by $M_\Gamma=\diag(\Id_4,m\,\Id_4 )$ and the central configuration $\vxi= \trasp{[1,0,-1-0,0,y,0,-y]}$, we immediately get that 
\begin{equation}\label{eq:mcsi-csi-romboa}
	\left\langle M_\Gamma\vxi,\vxi\right\rangle= 3m-3m^{2}+m\sqrt{9\left(1-m\right)^{2}+4m}+2
\end{equation}
%
%
%
%  
%\begin{align*}
%\left\langle M_\Gamma\vxi,\vxi\right\rangle  & =\begin{bmatrix}1\\
%0\\
%-1\\
%0\\
%0\\
%y\\
%0\\
%-y
%\end{bmatrix}^{T}\cdot\textrm{diag}[ 1,1,1,1,m,m,m] \cdot\begin{bmatrix}1\\
%0\\
%-1\\
%0\\
%0\\
%y\\
%0\\
%-y
%\end{bmatrix}\\
% & =2(my^{2}+1)
%  =2\left[m\frac{1}{2}\left(3\left(1-m\right)+\sqrt{9\left(1-m\right)^{2}+4m}\right)+1\right]\\
% & =.
%\end{align*}
Define the function $r(m)\=3m-3m^{2}+m\sqrt{9\left(1-m\right)^{2}+4m}+2$ and in order to find the root of the equation $r(m)=0,$ we first compute the roots of the equation 
\[
\left(3m-3m^{2}+2  \right)^2-\left(m\sqrt{9\left(1-m\right)^{2}+4m}  \right)^2=0.
\]
 By a simple calculation we get s that 
		\begin{multline}
		 \left(3m-3m^{2}+2\right)^{2}-\left(m\sqrt{9\left(1-m\right)^{2}+4m}\right)^{2}\\
		=  9m^{4}-18m^{3}-3m^{2}+12m+4
		 -\left(9m^{4}-14m^{3}+9m^{2}\right)\\
		=  -4m^{3}-12m^{3}+12m^{2}+4
		=  -4\left(m-1\right)\left(m+2+\sqrt{3}\right)\left(m+2-\sqrt{3}\right).
		\end{multline}
		This implies that $\left(3m-3m^{2}+2  \right)^2-\left(m\sqrt{9\left(1-m\right)^{2}+4m}  \right)^2=0$ has three real roots given by $-2-\sqrt{3}$, $-2+\sqrt{3}$ and $1$. Moreover, $r(1)=4>0,r(-2+\sqrt{3})=0$ and $r(2+\sqrt{3})=-50-30\sqrt{3}<0,$ and then the function $m \mapsto r(m)$ has the unique root given by $-2+\sqrt{3}.$     
 So we have that
\begin{equation}\label{eq:rhombus A z0Mz0}
\begin{cases}
\left\langle M_\Gamma\vxi,\vxi\right\rangle > 0, & \textrm{ if }\ -2+\sqrt{3}<m\le1\\
\left\langle M_\Gamma\vxi,\vxi\right\rangle <0, & \textrm{ if }-2-\sqrt 3<m<-2+\sqrt{3}.
\end{cases}
\end{equation}
\begin{itemize}
	\item {\bf Linearly stable relative equilibria.\/} 	If $-2+\sqrt{3}<m\leq 1$ directly by  Theorem \ref{1.Rhombus stable}  we get that  $\vz$ is linearly stable.	So by Theorem \ref{thm:stability-Morse-index-intro}, by Equation \eqref{eq:rhombus A z0Mz0} and by remembering that  fin this case the angular velocity $\omega$ is always  positive we conclude that

	 \begin{equation}
		\iMor(\widehat A_\Gamma)=\begin{cases}
		n_-\left(M_\Gamma \right)=0 & \textrm{ if } \ 0<m\leq 1\\
		n_-\left(M_\Gamma \right)-1=3 & \textrm{if }\ -2+\sqrt{3}<m<0.
		\end{cases}
		\end{equation}  
\item{\bf Unstable relative equilibria.\/} By invoking once again  the first claim in  Theorem \ref{1.Rhombus stable}, we know that if  $-1<m<-2+\sqrt{3}$, then the relative equilibrium 
$\vz$ is unstable and the nontrivial eigenvalues consist of a 
pair of real and a pair of purely imaginary eigenvalues. By Lemma \ref{thm:spectrum-B-Hessian} it follows that the non-trivial eigenvalues of $B$ are $\pm\sqrt{-\omega^2+\mu^2_1 }$ and $\pm\sqrt{-\omega^2+\mu^2_2 }$. Moreover, for  $-1<m<-2+\sqrt{3}$, author in \cite{Rob13} proved that  $ \sqrt{-\omega^2+\mu^2_1 }\in \R $  and $\sqrt{-\omega^2+\mu^2_2 } \in i \R$. Summing up, we get  that $|\mu_1|>\omega $ whilst $|\mu_2| <\omega$.	

Being $	|\mu_{1}|>\omega,$ we have that $(\omega+\mu_1)(\omega-\mu_1)<0$ for $m\in(-1,-2+\sqrt{3})$ which means that the factors $\omega \pm  \mu_1$ have opposite signs.
By Equation \eqref{eq:Rhombus-general-eigenvector-subspace}
\[
\left.\widehat{A_\Gamma}\right|_{I_{\omega+\mu_{1}}}[\vv_1]=(\omega+\mu_{1})\left\langle M_\Gamma \vv_1,\vv_1\right\rangle \quad \textrm{ and  }\quad
\left.\widehat{A_\Gamma}\right|_{I_{\omega-\mu_{1}}}[\vv_1]=(\omega-\mu_{1})\left\langle M_\Gamma \vv_1,\vv_1\right\rangle.
\]
Each of the  forms $\left.\widehat{A_\Gamma}\right|_{I_{\omega\pm\mu_{1}}}$ is a quadratic form onto a one-dimensional space and by the previous discussion on signs, we get that  $\iMor\left(\left.\widehat{A_\Gamma}\right|_{I_{\omega+\mu_{1}}}\right)+\iMor\left(\left.\widehat{A_\Gamma}\right|_{I_{\omega-\mu_{1}}}\right)=1$  for $m\in(-1,-2+\sqrt{3})$.
By a straightforward calculation, we also get that $\left\langle M_\Gamma\vv_{2},\vv_{2}\right\rangle  = 2m^{2}+2m$ where $\vv_2$ was  given in Equation \eqref{eq:v_1ev_2}. We observe that $\left\langle M_\Gamma\vv_{2},\vv_{2}\right\rangle$ is negative  for $m \in (-1,0)$ and a fortiori for $m \in (-1, -2+\sqrt 3)$. Now since $|\mu_2| <\omega$, this  implies that $(\omega+\mu_2)(\omega-\mu_2)>0$.  However, by the definition of $\mu_2$ given in Equation \eqref{eq:mu1-mu2}, we infer that $\mu_2<0$ and in particular being $\omega >0$ we get that $\omega -\mu_2 >0$. Thus by the product rule it then follows that also $\omega +\mu_2 >0$
Since, 
\[
\left.\widehat{A_\Gamma}\right|_{I_{\omega+\mu_{2}}}[\vv_2]=(\omega+\mu_{2})\left\langle M_\Gamma \vv_2,\vv_2\right\rangle \quad \textrm{ and  }\quad
\left.\widehat{A_\Gamma}\right|_{I_{\omega-\mu_{2}}}[\vv_2]=(\omega-\mu_{2})\left\langle M_\Gamma \vv_2,\vv_2\right\rangle.
\]
we finally get both quadratic forms (on the one-dimensional subspace generated by $\vv_2$) are negative definite and hence each gives a $1$ contribution to the Morse index. In conclusion, we get 
\begin{align*}
\iMor\left(\widehat{A}_\Gamma\right) & =\iMor\left(\left.\widehat{A_\Gamma}\right|_{I_{0}}\right)+\iMor\left(\left.\widehat{A_\Gamma}\right|_{I_{2\omega}}\right)+\iMor\left(\left.\widehat{A_\Gamma}\right|_{I_{\omega}}\right)+\iMor\left(\left.\widehat{A_\Gamma}\right|_{I_{\omega+\mu_{1}}}\right)\\
& \ +\iMor\left(\left.\widehat{A}\right|_{I_{\omega-\mu_{1}}}\right)+\iMor\left(\left.\widehat{A}\right|_{I_{\omega+\mu_{2}}}\right)+\iMor\left(\left.\widehat{A}\right|_{I_{\omega-\mu_{2}}}\right)\\
& =0+1+0+1+1+1\\
& =4.
\end{align*}
\end{itemize}

\paragraph{Rhombus B central configuration.} We assume now that the relative equilibrium  $\vz_{0}$ comes out from rhombus B central configuration. Once again by using Theorem
\ref{1.Rhombus stable},  it follows that $\vz_{0}$ is always unstable for $m\in(-1,0),$
moreover, we know that 
\begin{equation}\label{eq:change-sign-omega}
\begin{cases}
\omega >0 & \textrm { if } \quad  m \in (-2+\sqrt{3},0)\\
\omega<0 &  \textrm{ if }   \quad m \in (-1,-2+\sqrt{3}).
\end{cases}
\end{equation}
and one pair of eigenvalues is always real whilst the other  is purely imaginary for $-1<m<m^{*}$ and real
for $m^{*}<m<0$, where $m^{*}\approx-0.5951$ is the only real root
of the cubic $9m^{3}+3m^{2}+7m+5$. At $m=m^{*}$, rhombus $B$ is
degenerate. 

By invoking  Lemma \ref{thm:spectrum-B-Hessian}, the non-trivial eigenvalues of $B$ are $\pm\sqrt{-\omega^2+\mu^2_1 }$ and $\pm\sqrt{-\omega^2+\mu^2_2 }$. Furthermore by \cite[pag. 1129]{Rob13}, we know also that 	
	\begin{equation}
	\begin{cases}
	\sqrt{-\omega^{2}+\mu_{1}^{2}}&\textrm{ is real if } m\in(-1,0),\\
	\sqrt{-\omega^{2}+\mu_{2}^{2}}& \textrm{ is real if } m\in(m^{*},0),\\
	\sqrt{-\omega^{2}+\mu_{2}^{2}}&\textrm{ is purely imaginary if }m\in(-1,m^{*}).
	\end{cases}
	\end{equation}
	Thus we get  
	\begin{equation}
		\begin{cases}
	|\mu_{1}|>|\omega|, & \text{if }m\in(-1,0),\\
	|\mu_{2}|>|\omega|, & \text{if }m\in(m^{*},0),\\
	|\mu_{2}|<|\omega|, & \text{if }m\in(-1,m^{*}).
	\end{cases}\label{eq:rhombus B the internal for non-trivial eigenvalues}
	\end{equation}

		 Being $	|\mu_{1}|>|\omega|,$ we have that $(\omega+\mu_1)(\omega-\mu_1)<0$ for $m\in(-1,0)$.
	By Equation \eqref{eq:Rhombus-general-eigenvector-subspace}
	\[
	\left.\widehat{A_\Gamma}\right|_{I_{\omega+\mu_{1}}}[\vv_1]=(\omega+\mu_{1})\left\langle M_\Gamma \vv_1,\vv_1\right\rangle \quad \textrm{ and  }\quad
	\left.\widehat{A_\Gamma}\right|_{I_{\omega-\mu_{1}}}[\vv_1]=(\omega-\mu_{1})\left\langle M_\Gamma \vv_1,\vv_1\right\rangle
	\]
	 then we can conclude  that  $\iMor\left(\left.\widehat{A_\Gamma}\right|_{I_{\omega+\mu_{1}}}\right)+\iMor\left(\left.\widehat{A_\Gamma}\right|_{I_{\omega-\mu_{1}}}\right)=1$.
		By the very some argument, we infer also that $\iMor\left(\left.\widehat{A_\Gamma}\right|_{I_{\omega+\mu_{2}}}\right)+\iMor\left(\left.\widehat{A_\Gamma}\right|_{I_{\omega-\mu_{2}}}\right)=1 $ for $m\in(m^*,0)$. Next, we will compute the sign of 	$
		\left.	\widehat{A_\Gamma}\right|_{I_{\omega\pm\mu_{2}}} $ for $m\in(-1,m^*)$.    Since $\left\langle M_\Gamma\vv_{2},\vv_{2}\right\rangle =2m^{2}+2m $ which is negative in $(-1,0)$, a fortiori it will be negative in for $m\in(-1,m^*)$. Since $-2+\sqrt{3}\approx-0.2679>-0.5951\approx m^*$,   then  $\omega<0$ for $(-1,m^*)$. Now, since $-\omega^2+\mu_2^2 <0$ we get that   $(\mu_2-\omega)(\mu_2+\omega)<0$ and since $\mu_2 +\omega<0$ (in fact $\omega$ is negative for $m \in (-1, m^*)$ as well as $\mu_2$), this implies that $\mu_2-\omega >0$. Thus in conclusion both eigenvalues  $\mu_2 \pm \omega$ are negative in $(-1,m^*)$.  By Equation \eqref{eq:Rhombus-general-eigenvector-subspace}, we have  that  
		 \[
	\left.\widehat{A_\Gamma}\right|_{I_{\omega+\mu_{2}}}[\vv_2]=(\omega+\mu_{2})\left\langle M_\Gamma \vv_2,\vv_2\right\rangle \quad \textrm{ and  }\quad
	\left.\widehat{A_\Gamma}\right|_{I_{\omega-\mu_{2}}}[\vv_2]=(\omega-\mu_{2})\left\langle M_\Gamma \vv_2,\vv_2\right\rangle
	\]
		 and immediately by the discussion above, we conclude that 
		 \[
\iMor\left(\left.\widehat{A_\Gamma}\right|_{I_{\omega+\mu_{2}}} \right)=n_-\left(\left.\widehat{A_\Gamma}\right|_{I_{\omega-\mu_{2}}} \right)=0 \quad \textrm{ for }\quad m \in (-1, m^*).
\]
Invoking once again Equation \eqref{eq:Rhombus-general-eigenvector-subspace}, we get that $\left.\widehat{A_\Gamma}\right|_{I_{2\omega}}[K\vxi]=2\omega\left\langle M_\Gamma K\vxi,\vxi \right\rangle$. Summing up Equation (\ref{eq:rhombus A z0Mz0}) and Equation \eqref{eq:change-sign-omega},  then we get
\[
\iMor\left(\left.\widehat{A_\Gamma}\right|_{I_{2\omega}} \right)=0\ 
\textrm{for all }\ m\in(-1,-2+\sqrt{3})\cup(-2+\sqrt{3},0) 
\] 
By Equation (\ref{eq:Rhombus-general-eigenvector-subspace}), we get finally that  
\[
	\iMor\left(\widehat{A_\Gamma}\right)=\begin{cases}
	1+1=2 & \textrm{ if }\quad -2+\sqrt{3}<m<0,\\
	2+1+1=4 & \textrm{ if }\quad m^{*}<m<-2+\sqrt{3},\\
	2+1=3 & \textrm{ if }\quad -1<m<m^{*}.
	\end{cases}
	\] 
This concludes the proof. 
\end{proof}

 % % % % ================================================================

\vspace{3cm}
\noindent
\textsc{Prof. Xijun Hu}\\
Department of Mathematics\\
Shandong University\\
Jinan, Shandong, 250100\\
The People's Republic of China \\
China\\
E-mail:\email{xjhu@sdu.edu.cn}

\vspace{1cm}
\noindent
\textsc{Prof. Alessandro Portaluri}\\
DISAFA\\
Università  degli Studi di Torino\\
Largo Paolo Braccini 2 \\
10095 Grugliasco, Torino\\
Italy\\
Website: \url{aportaluri.wordpress.com}\\
E-mail: \email{alessandro.portaluri@unito.it}

\vspace{1cm}
\noindent
\textsc{Dr. Qin Xing}\\
Department of Mathematics\\
Shandong University\\
Jinan, Shandong, 250100\\
The People's Republic of China \\
China\\
E-mail:\email{qinxingly@gmail.com}

\end{document}